\documentclass[10pt]{amsart}

\usepackage[T1]{fontenc}
\usepackage[utf8]{inputenc}
\usepackage[british]{babel}

\usepackage{amsmath,amssymb,amsthm,amsfonts}
\usepackage{mathtools}
\usepackage{xargs}
\usepackage{dutchcal}
\usepackage[normalem]{ulem}
\usepackage[a4paper,margin=1in]{geometry}
\usepackage[colorlinks=true,linkcolor=blue,citecolor=blue,urlcolor=blue]{hyperref}

\theoremstyle{plain}
\newtheorem{theorem}{Theorem}[section]
\newtheorem{lemma}[theorem]{Lemma}
\newtheorem{proposition}[theorem]{Proposition}
\newtheorem{corollary}[theorem]{Corollary}

\theoremstyle{definition}
\newtheorem{remark}[theorem]{Remark}

\newtheorem{definition}[theorem]{Definition}

\DeclareMathOperator{\Div}{div}
\newcommand{\R}{\mathbb{R}}

\newcommand{\Z}{\mathbb{Z}}
\newcommand{\N}{\mathbb{N}}
\newcommand{\g}{\gamma}
\newcommand{\comp}[1]{{#1}^\complement}
\DeclareMathOperator{\dist}{dist}

\newcommand{\e}{\varepsilon}
\newcommand{\ud}{\,\mathrm{d}}
\newcommand{\sd}{\,\mathrm{sd}}
\newcommand{\ez}{\e\Z^N}

\usepackage[maxbibnames=99]{biblatex}
\numberwithin{equation}{section}

\title[Discrete diffusion/redistancing  for the mean curvature flow]{Elementary discrete diffusion/redistancing schemes for the mean curvature flow}
\author[A. Chambolle]{Antonin Chambolle}
\address[Antonin Chambolle]{CEREMADE, CNRS, Universit\'e Paris-Dauphine, PSL, and Mokaplan, INRIA Paris, France.}
\email[A. Chambolle]{antonin.chambolle@ceremade.dauphine.fr}

\author[D. De Gennaro]{Daniele De Gennaro}
\address[Daniele De Gennaro]{Department of Decision Sciences and BIDSA, Bocconi University, Milano, Italy}
\email[D. De Gennaro]{daniele.degennaro@unibocconi.it}
\author[M. Morini]{Massimiliano Morini}
\address[Massimiliano Morini]{Dip.~di Matematica, Univ.~Parma, Italy}
\email[M. Morini]{massimiliano.morini@unipr.it}
 
\date{}

\begin{document}

\begin{abstract}
   We consider a fully discrete and explicit scheme for the mean curvature
    flow of boundaries, based on an elementary diffusion step and a precise redistancing operation. We give an elementary convergence proof for
    the scheme under the standard CFL condition $h\sim\e^2$, where $h$ is the
    time discretization step and $\e$ the space step. We discuss extensions
    to more general convolution/redistancing schemes.
\end{abstract}
\maketitle
\section{Introduction}
We analyze a basic \textit{explicit} fully discrete (in space and time) scheme for computing the mean curvature
flows of curves and surfaces which are boundaries of sets. This approach belongs to
the family of convolution-generated motions, yet it is much more elementary. It
consists in alternating an explicit approximation of the heat equation, with time-step $h>0$, with a redistancing
operation on rescaled regular grids $\ez$, for $\e>0$.  

It is heuristically well known  that this type of approach is effective, and variants of the corresponding algorithm have been proposed many times in the literature as a natural, more accurate extension of the celebrated Merriman-Bence-Osher (MBO in short) scheme~\cite{BMO92}. See for instance \cite{KimNot}, where a discrete, finite-elements based approach is implemented using  a brute-force redistancing (as in our work) and  experimented on several interface evolution problems, and the first Algorithm in~\cite{EsedogluRuuthTsai10}. Rigorous convergence results and error estimates for this algorithm can be found in~\cite{IshKim16,IshIzu18}. We also refer to~\cite{EsedogluRuuthTsai08,Metivetetal2021}, where similar techniques are primarily used to simulate higher-order geometric flows, and to~\cite{ElseyEsedogluSmereka11}, which implements a method introduced and analyzed in~\cite[Sec.~5.3]{EsedogluRuuthTsai10} for the simulation of large-scale multiple-grain evolution in materials.

Indeed, as clearly pointed out in \cite{EsedogluRuuthTsai10}, despite its computational efficiency, the MBO algorithm has some well-known limitations. In particular, if the spatial discretization is not refined alongside the time step, pinning phenomena may appear. More generally,  significant discretization errors may arise in the computed evolution. To address these issues, it is crucial to use a discretization method that allows for a more precise representation of the interface position within the grid. 
Unlike the characteristic functions used in the original MBO scheme, signed distance functions can be represented more accurately on uniform grids, even at subgrid resolutions, thanks to their Lipschitz continuity.

The algorithm we propose aims to address these issues by combining a diffusion step with a redistancing one, as proposed in~\cite{KimNot,EsedogluRuuthTsai10} ---  in a fully discrete formulation. In the first part of the paper, instead of considering the MBO-type diffusion step, based on the discrete heat kernel, we  directly consider the Euler explicit scheme for the discrete heat equation. The redistancing step also changes, as we now use the redistancing operator we recently introduced in~\cite{CDGM-crystal}.
In our first main result, Theorem \ref{th:mainexplicit}, we show the convergence of such a fully discrete scheme  towards a generalized solution to the mean curvature flow. 
We then generalize this result in two directions. From Section  \ref{sec:kernels} on, we modify the diffusion step to allow for general kernels $K^\e$ depending on the space discretization parameter $\e$ and satisfying natural assumptions. Precisely, the kernels $K^\e$ are required to be  non-negative, symmetric and of mass one. Moreover, their second moment (which is assumed to be vanishing as $\varepsilon \to 0$) intrinsically defines the time step $h$ as a function of the  space step $\varepsilon$ (see \eqref{Kh}). Finally, the kernels are assumed to suitably approximate the Laplacian. Such an approximation property is imposed  in terms of the pointwise convergence of the discrete Fourier transform of the kernels to  the symbol of the Laplacian (see \eqref{eq:approxsymbol}),  which  allows for straightforward verification in many practical examples, as detailed in Section \ref{sec:applic}. Under these assumptions, we can show that our diffusion/redistancing scheme converges, see Theorem \ref{th:mainimplicit}.

We then study a modification of the discrete redistancing operator we defined in \cite{CDGM-crystal} to account for a nonlinearity $\g$ in the scheme. The idea is that, in some applications, it may be more interesting to consider evolving profiles of the form $\g(\sd)$ rather than the signed distance function itself, for instance 
in the context of  Allen-Cahn-based approximation schemes for the mean curvature flow \cite{Alberti96, DeepLearningMCF}. We show how to design such a scheme in the instance where $\g$ is the Modica-Mortola
(or Cahn-Hilliard) optimal profile. We refer to Section \ref{sec:nonlinear} for a more detailed explanation of the heuristic.  Considering the general kernels $K^\e$ previously introduced for the diffusion step, we are able to show in  Theorem \ref{th:mainnonlin} that  the proposed  
diffusion/nonlinear redistancing scheme converges. This seems to give  a partial explanation for the good results produced by fully learned approaches for the mean curvature flow, introduced in~\cite{DeepLearningMCF}.

We are not aware of any previous rigorous convergence result for a fully discrete  similar
scheme towards  mean curvature flows, as the steps go to zero,
with precise bounds on the possible ratio between the time and space steps, even in the basic instance of the very simple Euler explicit scheme. 
Most of the previously mentioned references
consider a time-discrete setting, but continuous in
space, and then assume that if the space step is small enough, one will approximate the continuous process.
The point is that analyzing the convergence of a fully discrete
approximation requires a quite precise control
of the redistancing operation, which is in general unavailable.
There are, however, some works addressing this issue \cite{MisiastYip16, LauxLelmi23} for MBO-type algorithms (see also \cite{LauUll} for a similar study in the level-set framework). In the first reference, a study of convergence of the fully discrete MBO scheme is presented in dimension 2, showing  convergence  towards viscosity solutions to the mean curvature motion  under the assumption that $\e=o(h)$, while pinning is present whenever $h=o(\e)$. The second reference studies the instance of general graphs, supported on a fixed submanifold $M$. The authors consider MBO-type algorithms where the convolution kernels approximate
the heat kernel corresponding to a weighted Laplace-Beltrami operator on  $M$. The convergence of this scheme towards  mean curvature flows on the manifold $M$ is proved under some technical assumptions on the kernels, and for $\e = o(h^{\frac 32})$. 
Last, in the aforementioned work~\cite{KimNot}, a fully discrete
finite-element approximation was introduced
and experimentally justified for a few geometric flows. 
It would be interesting to check whether
a theoretical analysis similar to ours could  be possible for that approach, at least for basic flows
which enjoy a comparison principle. However, this is
not straightforward, in particular since the discretization
in~\cite{KimNot} is not translational invariant.

In our case, by introducing a particular form of (brute force)
redistancing (actually,
already introduced in~\cite{CDGM-crystal} for the study of discrete crystalline
curvature flows), defined by means of inf and sup-convolution
formulas,
we address this issue in a simple way and establish convergence
of various (similar) approaches, with {vanishing spatial
and time discretization steps determined by} the
convergence to the heat equation of the convolution step.
{Our method relies strongly on the comparison principle and on the standard ``stability-monotonicity-consistency'' framework for approximation schemes, which ensures convergence in similar settings; see, for instance, \cite{RouyTourin,BarlesBronsardSouganidis1992,BarSonSou}. In our case, monotonicity is elementary. Moreover, since the redistancing step based on inf/sup-convolutions is essentially exact, consistency follows in an almost straightforward way (using Lemma~\ref{lem:defsd}). As a consequence,  showing  stability is the only step that requires additional (but not particularly refined) bounds, which need not be sharp in the present context. On the other hand, the approach
does not yield
error estimates, which would require a much more precise asymptotic analysis of the flow near the   boundaries, at least in the smooth case.}

Our definition of a redistancing operator is computationally more expensive than
standard PDE based approach computed using fast-marching type
technique~\cite{RouyTourin,Set,OshSet,ElseyEsedoglu14,Saye14}, yet whose
definition make less clear the actual error estimates
and regularity properties of the distance that they compute  (see Section \ref{sec:strips} for a practical method to reduce the computational cost). Our redistancing operation is designed to compute a ``true distance'' on lattice points. In this sense, it differs from the standard graph distance, which measures distances along the edges of a prescribed graph. Equivalently, our construction can be interpreted as a graph distance on a complete graph. For this reason, one cannot expect standard efficient methods such as fast--marching or Hamilton--Jacobi solvers to provide a significant computational advantage over our  brute-force approach. 
As we show in \cite[Appendix~{B}]{CDGM-crystal}, in certain cases (for instance when the weights are rational, see \cite[Corollary~{B.2}]{CDGM-crystal}), the problem can be reduced to a finite-range interaction. In such situations, dynamic programming techniques may be employed to improve computational efficiency (see, e.g., \cite{Set,Tsitsiklis}). 
This problem is also closely related to the framework of Lipschitz learning on graphs \cite{RoiBun, CalderEttehad2022}. At its core, it consists in extending a function prescribed on a subset (here the zero sublevel set) to a $1$-Lipschitz function on the whole space in a maximal way. We refer to Section~\ref{sec:base_algo} for further details.

In a previous contribution~\cite{CDGM-crystal}, we were studying an implicit
approach for the crystalline mean curvature flow and we could show that, in some situation,
one could completely uncouple the time and space steps and
still get consistency of the scheme. This is not the case here,
as any redistancing still produces some (spatial) error which can
accumulate (in time) as the steps go to zero, if the time-step
is not always larger than the spatial error of the redistancing, that is  
$h\gtrsim \e^2$.
We show however that this is the only limitation.

The paper is structured as follows. 
We first study  in Sections~{\ref{sec:base_algo}-\ref{sec:conv_expl}} the discrete   
diffusion/redistancing scheme
built upon the standard discretization of the Laplace operator on a grid.
Then, in Section~\ref{sec:kernels} we generalize the convergence result to rather general symmetric,
non-negative convolution kernels.
In Section~\ref{sec:nonlinear} we introduce and study the instance of nonlinear redistancing operators.
We show some examples of kernels fitting in our framework (as the heat kernel, or the kernel associated to the implicit Euler's scheme) in Section~\ref{sec:applic} and some numerical experiments in Section~\ref{sec:numerics}. 
In Section~\ref{sec:concluding}, we provide some concluding remarks.

\section{A simple algorithm}\label{sec:base_algo}
\subsection{Laplacian operator and redistancing on a regular discrete grid.}

\subsubsection*{Laplacian}
The first scheme we present is based on the resolution of an \textit{explicit} scheme for the
heat flow, based on a standard discretization of the Laplacian. We will see
in Section~\ref{sec:kernels} how to generalize to other convolution
or averaging type operators.

Given $v=(v_i)_{i\in\ez}$ a real-valued function on the discrete grid $\ez$, $N\ge 1$,
we introduce the discrete Laplace operator:
\[
  ( \Delta_\e v)_i : = \frac{1}{\e^2}\sum_{n=1}^N (v_{i+\e e_n} - 2v_i + v_{i-\e e_n})
  = \frac{1}{\e^2}\sum_{n=1}^N (v_{i+\e e_n} + v_{i-\e e_n}) - \frac{2N}{\e^2}v_i.
\]
It is well known, and straightforward to show that if $\eta$ is a smooth ($C^2$) function
defined near a point $x\in\R^N$ and  {$\eta^\e_i:=\eta( i)$ for $i\in\ez$}, then $(\Delta_\e \eta^\e)_{\e i}\to \Delta \eta(x)$
as $\e\to 0$, whenever $\e i\to x$.
\subsubsection*{Redistancing}
As in~\cite{EsedogluRuuthTsai10}, the scheme will alternate a step of the discrete heat-flow and a redistancing operation. The redistancing is a general and important
issue in the approximation of geometric flows, and there are many ways to
address it, usually based on the resolution of a Hamilton-Jacobi equation,
either static or evolutionary~\cite{RouyTourin,HamamukiNtovoris16}.
However, in general, the method is studied in the continuous
setting and one assumes that the spatial discretization is
fine enough so that the scheme which is used approximates
well enough the continuum limit.

We analyze here the global convergence of such schemes as the
space and time steps go to zero. In that case, most approximate
redistancing schemes  {do} not seem robust enough to yield 
consistency of the evolution with the continuous limit.
In some sense, we need to compute an ``exact'' discrete
distance to the evolving set.
This is why we propose, as in our previous work~\cite{CDGM-crystal}, an inf/sup-convolution
based approach.

In practice, our method takes a $1$-Lipschitz\footnote{ 
Note that the definition itself does not require $u$ to be 1-Lipschitz. However, without this assumption, Lemma~\ref{lem:defsd}, which is essential to our proof strategy, fails to hold.} function over $\ez$ and returns a proxy for
the distance to the zero level set, which takes into account the values of the initial function and  
keeps a (more or less good) sub-pixel accuracy.
Assuming that $(u_i)_{i\in \ez}$ is $1$-Lipschitz:
\[
  |u_i-u_j|\le |i-j|\quad\forall i,j\in\ez,
\]
 we define the redistancing functions $\sd^\pm[u]$ by letting for $i\in\ez$:
\begin{equation}\label{eq:sdp}
  \begin{cases}
    \ud^+[u]_i := \inf_{j:u_j<0}  {(} u_j + |j-i| {)}\,,\\
    \sd^+[u]_i := \sup_{j:u_j\ge 0}  {(} \ud^+_j - |j-i|  {)}
  \end{cases}
\end{equation}
and
\begin{equation}\label{eq:sdm}
  \begin{cases}
    \ud^-[u]_i := \sup_{j:u_j>0}  {(}u_j - |j-i| {)}\,,\\
    \sd^-[u]_i := \inf_{j:u_j\le 0}  {(} \ud^-_j + |j-i| {)}\,.
  \end{cases}
\end{equation}
 {We use  the convention that $\inf \emptyset=+\infty, \sup \emptyset=-\infty$.}
When not ambiguous, we drop the dependence ``$[u]$'' in the notation.

The function $\sd^+$, $\sd^-$ can be seen as discrete ``signed distance functions''
to respectively the set $\{i:u_i<0\}$ and $\{i:u_i\le 0\}$.
In some sense, $\sd^+$, $\sd^-$ are respectively the largest
and smallest distance functions which we may use in our
scheme, and proving convergence
for these choices will yield convergence for any other choice 
which lies in between. A good choice
for a more precise distance is, for instance, to consider the average $(\sd^++\sd^-)/2$.  {The first elementary properties of the redistancing are shown in the following lemma, compare with  \cite[Section 4]{CDGM-crystal}.}

\begin{lemma}\label{lem:defsd}
 {Assume still that $(u_i)_{i\in\ez}$ is $1$-Lipschitz. Then}
  the distance function $\sd^+$ may also be defined as follows:
  \begin{enumerate} 
  \item $\ud^+$ is the largest $1$-Lipschitz function such that $\ud^+_i\le u_i$ when
    $u_i<0$, and in particular $\ud^+_i=u_i$ if $u_i<0$ and $\ud^+_i\ge u_i$ if $u_i\ge 0$,
  \item $\sd^+$ is the smallest $1$-Lipschitz function such that $\sd^+_i\ge \ud^+_i$
    when $u_i\ge 0$, and in particular $\sd^+_i=\ud^+_i\ge u_i$ if $u_i\ge 0$ and $\sd^+_i\le \ud^+_i=u_i$
    if $u_i<0$.
  \end{enumerate}
  We also deduce that $\sd_i^+\ge 0\Leftrightarrow u_i\ge 0$ and
  $\sd_i^+< 0 \Leftrightarrow u_i<0$.
\end{lemma}
\noindent A similar, symmetric statement holds for $\sd^-$, see \cite{CDGM-crystal}.  {We note that the definitions \eqref{eq:sdp}, \eqref{eq:sdm} are coherent and Lemma \ref{lem:defsd} holds even if $\{u<0\}=\emptyset$ (and thus $\sd^\pm[u]=+\infty$) or $\{u>0\}=\emptyset$ (and $\sd^\pm[u]= -\infty$).}

\begin{proof}
  First, clearly $\ud^+$ is $1$-Lipschitz as the inf-convolution of a function
  with $|\cdot|$:  {indeed,} given $i,i'\in\ez$ and  {$\sigma>0$}, if $j$ is such that
  $u_j<0$ and $\ud^+_i\ge u_j+|j-i|- {\sigma}$,
  then
  \[
     {\ud^+_{i'}-\ud^+_i} \le  u_j+|j-i'| - (u_j+|j-i|- {\sigma}) = |j-i'|-|j-i|+ {\sigma}\le |i-i'|+ {\sigma}
  \]
  by the triangle inequality.   {Letting $\sigma\to 0$ shows that $\ud^+$ is $1$-Lipschitz.} Also,
  $\ud^+_i\le u_i$ if $u_i<0$, choosing $j=i$ in the definition (actually,
  this is an equality, since if $j\neq i$ one has $u_i \le u_j + |j-i|$).

  Then, if $(v_i)_{i\in\ez}$ is $1$-Lipschitz and below $u$ where $u$ is negative, then
  for any $j$ with $u_j<0$,
  \[
    v_i \le v_j+|j-i| \le u_j+|j-i|,
  \]
  so that $v_i\le \ud^+_i$.  
  The proof of the second point of the claim is identical.
\end{proof}

Let us give some more details concerning the previous statement about ``signed distance function''. If $d_i = \min_{j:u_j<0} |i-j|$ is the distance to $\{j:u_j<0\}$,
then by definition $d \ge \ud^+$ and it is  {1}-Lipschitz, hence $d \ge \sd^+$.
On the other hand, let $j$ with $u_j<0$, $i$ with $d_i>0$ (hence $u_i\ge 0$),
and consider $J=\{j'\in \ez: \dist(j',[i,j])\le 2\sqrt{N}\e\}$,  {where $[i,j]\subset \R$ denotes the line segment between $i,j$}.
For $j'\in J$, $|i-j'| + |j'-j|\le |i-j|+4\sqrt{N}\e$ (considering  {an intermediate point $\tilde j'\in [i,j]\cap (j'+[0,2\e]^N)$}).  
If $u_{j'}<0$, one has $d_i \le |i-j'| \le |i-j|-|j-j'|+4\sqrt{N}\e$. We may take
for $j'$ the closest point in $\{j':u_{j'}<0\}\cap J$, so that $j'$ has a neighbor $j''$
(at distance $\e$) in $J$ with $u_{j''}\ge 0$, and $u_{j'}\ge -\e$.
Using that $|j-j'|\ge u_{j'}-u_{j}$, we find that
\begin{multline*}
  d_i \le |i-j'| \le |i-j|-|j-j'|+4\sqrt{N}\e
  \le  |i-j| + u_j - u_{j'}+4\sqrt{N}\e
  \\ \le u_j+|i-j| + (4\sqrt{N}+1)\e.
\end{multline*}
Taking the infimum over all possible $j$, we find that $d_i\le \ud^+_i+C\e = \sd^+_i+C\e$
since $u_i\ge 0$. We deduce
\begin{equation}\label{eq:compdist}
  \sd^+_i \le \dist(i,\{j:u_j<0\}) \le \sd^+_i+C\e,
\end{equation}
for $C=4\sqrt{N}+1$ and all $i$ with $u_i\ge 0$. Similarly, if $u_i<0$ then:
\begin{equation}\label{eq:compdistm}
  -\sd^+_i  \le \dist(i,\{j:u_j\ge 0\}) \le -\sd^+_i+C\e.
\end{equation}

\begin{proposition}[Comparison]
  Given $1$-Lipschitz functions  $(u_i)_{i\in\ez}$, $(u'_i)_{i\in \ez}$ ,
  and denoting respectively
  $\sd^\pm = \sd^\pm[u]$, $(\sd')^\pm=\sd^\pm[u']$
  the functions obtained by applying~\eqref{eq:sdp}, \eqref{eq:sdm}  to $u$, $u'$, we have:
  \begin{enumerate}
  \item $\sd^-\le \sd^+$,
  \item For any $s\ge 0$, if $u\le u'-s$ then $\sd^\pm \le (\sd')^\pm-s$.
  \end{enumerate}
\end{proposition}
\begin{proof}
  For the first point, recall that  $\ud^- \le u$, hence   $\ud^+ \ge \ud^-$. But then for
  $i\in\ez$,
  \[
    \sd^-_i :=
    \inf_{j:u_j\le 0}  {(}  \ud^-_j + |j-i|  {)}
    \le \inf_{j:u_j< 0}  {(}\ud^+_j + |j-i|  {)}
    = \inf_{j:u_j < 0}  {(} u_j + |j-i|  {)} =  \ud^+_i,
  \]
  using that $\ud^+_j=u_j$ if $u_j<0$.
  Since $\sd^+_i=\ud^+_i$ if $u_i\ge 0$,
  we deduce that $\sd^-_i\le \sd^+_i$ when $u_i\ge 0$. Symmetrically, $ 
  \sd^+_i := \sup_{j:u_j\ge 0}  {(} \ud^+_j - |j-i|  {)} \ge  \sup_{j:u_j\ge 0}  {(} \ud^-_j - |j-i|  {)} \ge \ud^-_i$, 
  which coincides with $\sd^-_i$ whenever $u_i\le 0$.

  For the second point,
  we have that $\ud^+_i := \inf_{j:u_j<0}  {(} u_j + |j-i|  {)} \le \inf_{j:u_j<0}  {(} u'_j-s + |j-i|  {)}$,
  yet if $u'_j<0$, $u_j\le u_j'-s< -s\le 0$, hence $\{j:u'_j<0\}\subset \{j:u_j<0\}$
  and the inf is less than $\inf_{j:u'_j<0}  {(} u'_j + |j-i|-s  {)}= (\ud')^+_i-s$.
  Then, 
  $\sd^+_i := \sup_{j:u_j\ge 0}  {(} \ud^+_j - |j-i|  {)} \le\sup_{j:u_j\ge 0}  {(} (\ud')^+_j - |j-i|-s  {)} $
  and similarly, $u_j\ge 0\Rightarrow u'_j\ge s\ge 0$ and we deduce the claim.
\end{proof}

\subsection{Algorithm}
\label{sec:algo_explheat}
We consider the following algorithm.
Given an initial closed set $E^0\subset\R^N$, we set $E^{0,\e}:=\{i\in\ez: i+[0,\e)^N\cap  E^0\neq \emptyset\}$, and let $u^{0,\e}=\sd^{0,\e}$ be a $1$-Lipschitz function
on $\ez$ such that $\sd^{0,\e}_i < 0$  on $E^{0,\e}$, and $>0$ else. For all $k\in\N$ we define:
\begin{equation}\label{eq:heat}
  \frac{u_i^{k+1,\e}-\sd_i^{k,\e}}{h} = (\Delta_\e \sd^{k,\e})_i,
\end{equation}
where the time-step $h$ will be precised right after, as a function of $\e$.
We then  let
\begin{equation}\label{eq:redistance}
  \sd^{k+1,\e} = \sd^+[u^{k+1,\e}],
\end{equation}
applying~\eqref{eq:sdp} to $u=u^{k+1,\e}$ (alternatively, we could use \eqref{eq:sdm}).

We fix $\theta\in  {(0,1]}$ and let $h = \theta\frac{\e^2}{2N}$.
Then \eqref{eq:heat} becomes:
\[
  u_i^{k+1,\e} = (1-\theta) \sd_i^{k,\e} + \frac{\theta}{2N}\sum_{n=1}^N
  \left(\sd_{i+\e e_n}^{k,\e} + \sd_{i-\e e_n}^{k,\e}\right),
\]
where $(e_n)_{n=1}^N$ is the canonical basis of $\R^N$:
this guarantees  {(by induction)} that $u_i^{k+1,\e}$ is $1$-Lipschitz
as a convex combination of $1$-Lipschitz functions, and in particular Lemma~\ref{lem:defsd}
holds for all iterates. From this, we deduce that 
at points $i\in\ez$ where $u^{k+1,\e}_i\ge 0$ or equivalently $\sd^{k+1,\e}_i\ge 0$ it holds:  
\begin{equation}\label{eq:subsoldiscrete}
  \frac{\sd_i^{k+1,\e}-\sd_i^{k,\e}}{h} \ge (\Delta_\e \sd^{k,\e})_i,
\end{equation}
while where $u^{k+1,\e}_i<0 \Leftrightarrow \sd^{k+1,\e}_i< 0$,
\begin{equation}\label{eq:supersoldiscrete}
  \frac{\sd_i^{k+1,\e}-\sd_i^{k,\e}}{h} \le (\Delta_\e \sd^{k,\e})_i.
\end{equation}
We see that approximately, we have built a signed distance function
which is a supersolution of the heat equation where positive,
and a subsolution where negative, which is consistent with the
distance function to a set evolving by its mean curvature  {see \cite{Soner93}}.
Our first result is the following.
\begin{theorem}\label{th:mainexplicit}
  As $\e\to 0$, the function $(d^\e(t)_i)$, defined for $t\ge 0$ and $i\in\e\Z^N$
  by $d^\e(t)_i = \sd^{[t/h],\e}_i$ (where $[\,\cdot\,]$ denotes the
  integer part), converge up to subsequences,
  for almost all times and locally uniformly in space to a function
  $d(x,t)$ such that $d^+=\max\{d,0\}$ is the distance function
  to a supersolution of the (generalized)
  mean curvature flow starting from $E^0$,
  and $d^-$ is the distance function to a supersolution starting from
  $\comp{E^0}$. In particular, if the mean curvature flow $E(t)$ starting
  from $E^0$ is unique, then $d^\e(t)$ converges to the signed distance
  function to $E(t)$, up to extinction.
\end{theorem}
Here, by generalized solution, we mean a solution in the viscosity
sense, as defined in~\cite{EvaSpr,BarSou,BarSonSou,Soner93}, or, equivalently,
in the distributional sense introduced in~\cite{CMP17} for nonsmooth
flows. We make this precise in the next section, before
proving Theorem~\ref{th:mainexplicit}.
\section{Convergence of the algorithm}\label{sec:conv_expl}

\subsection{Notions of generalized mean curvature flow}

We consider the classical generalized sub/superflows which
have been defined since~\cite{EvaSpr} by viscosity solutions
or barriers~\cite{BellettiniNovaga97}, yet the characterization
which is best adapted for showing consistency of our schemes
is the one involving the (signed) distance function to the interface,
introduced in~\cite[Def.~5.1]{Soner93} (which essentially is saying
that the flow is characterized by the property that the distance
to the evolving boundary should be a supersolution of the heat flow
away from the boundary), and the distributional variant
in~\cite[Def.~2.1]{CMP17} which was introduced for the study of nonsmooth
flows (but it is equivalent in the smooth case). We recall the definition
in~\cite[Def.~2.1]{CMP17}
(see also~\cite[Def.~2.2]{CMNP19}), specified to the case of the standard mean curvature flow, which is
less common than~\cite{Soner93}.
\begin{definition}\label{def:distributionalsuperflow}
  Let $E^0\subset \R^N$ be a closed set. Let $E$ be a closed
  set in $\R^N\times [0,+\infty)$ and for each $t\ge 0$ denote
  $E(t):=\{x\in\R^N:(x,t)\in E\}$. We say that $E$ is a supersolution 
  of the mean curvature flow (in short, a superflow) with initial datum $E^0$ if\begin{itemize}
  \item[(a)] $E(0)\subseteq E^0$;
  \item[(b)] for all $t\ge 0$, if $E(t)=\emptyset$ then $E(s)=\emptyset$
    for all $s>t$;
  \item[(c)] $E(s)\stackrel{\mathcal{K}}{\rightarrow}E(t)$ as $s\uparrow t$ for all $t>0$
    (Kuratowski left continuity);
  \item[(d)] setting $d(x,t)=\dist(x,E(t))$ for $(x,t)\in(\R^N\times (0,T^*))\setminus E$ where
    \[ T^*:=\inf \{ t>0: E(s)=\emptyset \text{ for } s\ge t\},\]
    then the following inequality holds in the distributional sense in $(\R^N\times (0,T^*))\setminus E$
    \begin{equation}\label{eq:supsolheat}
      \frac{\partial d}{\partial t} \ge \Delta d.
    \end{equation}
    
  \end{itemize}
  A subsolution (in short, subflow) is the complement of a supersolution, and we say that $E$ is a solution if it is a supersolution while $\mathring{E}$ is a subsolution.
\end{definition}

\begin{remark}
Note that in definition \cite[Def.~2.2]{CMNP19}, an additional requirement on $\Delta d$ is imposed, namely:  $\Delta d$ is a Radon measure in $(\R^N\times (0,T^*))\setminus E$,     and $(\Delta d)^+ \in L^\infty(\{(x,t)\in\R^N\times (0,T^*):
d(x,t)\ge \delta\})$ for every $\delta\ge 0$. However, when $d$ is the distance induced by the Euclidean norm, this condition is automatically satisfied, thanks to the following standard semiconcavity estimate. Let $d(x,t)\ge \delta>0$,  $y$ be the projection of $x$ on $\partial E(t)$ and $h$ small enough ($|h|\le \delta/2$). Then 
\begin{multline*}           
    d(x+h,t)-2d(x,t)+d(x-h,t)\le |x+h-y| - 2|x-y|+ |x-h-y|\\
    = \int_{-1}^1 (1-| {s}|)\frac1{|x-y+ {s}h|}\Big(I-\frac{(x-y+ {s}h)\otimes (x-y+ {s}h)}{|x-y+ {s}h|^2}\Big) h\cdot h \, d {s}\\
    \le      c|h|^2 \int_{-1}^1 \frac{1-| {s}|}{ {|x-y+ {s}h|}} \, d {s}
    \le \frac{c}{\delta}|h|^2,
\end{multline*}
{where we used a Taylor expansion.
From this, it}
follows that the distributional  second  spatial derivatives of $d$ are uniformly bounded above by $\frac c\delta$ in the region $\{d\ge \delta\}$.
\end{remark}

It has been proven in \cite{CMP17} (see also \cite{ChaMorNovPon19}) that generically (up to fattening) there is a unique solution starting from a set $E^0$;
in addition, it is proved in~\cite[Appendix]{CMP17} that~\eqref{eq:supsolheat} holds in the viscosity sense
and that this definition is equivalent to~\cite[Def.~5.1]{Soner93}.

We prove that the sequence of distance functions produced
by the algorithm defines, in the limit, generalized solutions
 {in  this} sense.
This will follow from (i) an estimate on the motion of a ball which yields enough
control on $d^\e$ in time to show compactness, (ii) an elementary
consistency argument. We point out that the only true difficulty for generalizing
this result is point~(i),
which is where the precision of the redistancing plays a major role.
 
\subsection{Evolution of a ball}

As said, a crucial point for proving the convergence of the method, is to control the behavior
of the algorithm when $\sd^{k,\e}_i$ represents a ball of radius $R>0$. For this,
given $\theta\in  {(0,1]}$, we let for $i\in\ez$,
\begin{equation}\label{eq:v}
    u_i= |i|-R,\quad v_i = (1-\theta) u_i + \frac{\theta}{2N} \sum_{n=1}^N (u_{i+\e e_n}+u_{i-\e e_n})   
\end{equation}
and let $d=\sd^+[v]$. Then, we show:
\begin{lemma}\label{lem:balls}
  There exists a dimensional constant $C\ge 1$ such that if $\e/R$ is small enough (depending
  only on the dimension), then for every $i\in\ez$
  \[
    d_i \le |i|-R + \frac{C}{R}\e^2 = |i|-R + \frac{2NC}{\theta R}h  
  \]
\end{lemma}
\noindent Here as before $h,\e,\theta$ are linked by the relationship $\theta=2Nh/\e^2$.
\begin{proof}
  Without loss of generality we assume $\theta=1$. 
  We first observe that
  \[
    v_i = -R + \frac{1}{2N}\sum_{n=1}^N |i+\e e_n|+ |i-\e e_n|.
  \]
  We remark that if $|i|\ge 2\e$, $e\in\{e_n,-e_n:n= 1,\dots ,N\}$,
   {by a   Taylor expansion it holds:}
  \begin{multline}\label{eq:Taylor}
    |i+\e e| = |i| + \e\frac{i\cdot e}{|i|} +
    \int_0^1 (1-t) \frac{1}{|i+t\e e|}\left(I - \frac{(i+t\e e)\otimes (i+t\e e)}{|i+t\e e|^2}
    \right)\e e \cdot \e e\,dt  \\ 
    \le  |i| + \e\frac{i\cdot e}{|i|} + \frac{1}{|i|-\e}\frac{\e^2}{2}
  \end{multline}
  so that:
  \[
    v_i \le |i|-R + \frac{1}{|i|-\e}\frac{\e^2}{2}, \qquad  {\text{for } |i|\ge 2\e}.
  \]
  If we assume that $\e\le \min\{1,R/2\}$, then for $|i|\le R-\e$,
  $v_i< 0$. Hence, since by Lemma~\ref{lem:defsd}
  it is enough
  to estimate $d_i^+$ at points
  $i\in\ez$ with $v_i\ge 0$, we assume  $|i|\ge R-\e$. For all such $i$ we have:
  \begin{equation}
  \begin{split}\label{eq:d+est}
      \ud^+[v]_i :&= \inf_{j:v_j<0}  {(} v_j + |j-i|  {)}
    \le \inf_{j: \frac{R}{2}+\e\le |j|\le R-\e}  {(} v_j + |j-i|  {)} \\
    &\le \inf_{j: \frac{R}{2}+\e \le |j|\le R-\e}  {(} |j|-R + \frac{\e^2}{R} + |j-i|  ),
  \end{split}
  \end{equation}
  assuming $\e\le R/8$ so that the set of $j$'s is not empty.
  Consider $j$ with $\frac{R}{2}+\e \le |j|\le R-\e$, close to the
  segment $[0,i]$:  if $\tilde{\jmath}$ is the projection of $j$ onto $[0,i]$, one has
  \begin{multline*}
    |j|+|j-i| =  \sqrt{ |j-\tilde{\jmath}|^2 +|\tilde{\jmath}|^2} + \sqrt{|\tilde{\jmath}-i|^2+|j-\tilde{\jmath}|^2}
    \\
    \le |\tilde{\jmath}| + \frac{|j-\tilde{\jmath}|^2 }{2|\tilde{\jmath}|}
    + |\tilde{\jmath}-i| + \frac{|j-\tilde{\jmath}|^2 }{2|\tilde{\jmath}-i|}
    = |i| + \left(\frac{1}{2|\tilde{\jmath}|} + \frac{1}{2|\tilde{\jmath}-i|} \right){|j-\tilde{\jmath}|^2 }
  \end{multline*}
  If $\e/R$ is small enough (depending only on $N$), we can find
  $j,\tilde{\jmath}$ such that $|j-\tilde{\jmath}|^2\le N\e^2$ and $R/2\le |\tilde{\jmath}|\le 3R/4$,
  so that $|\tilde{\jmath}-i|\ge R/4$. We obtain for such a choice:
\begin{equation}\label{eq:triangle}
    |j|+|j-i| \le |i|+\frac{3N}{R}\e^2
\end{equation}
  This shows that where $d^+_i$ is non-negative, it is
  less than $|i|-R + \frac{3N+1}{R}\e^2$ as soon as $\e/R$ is small enough. As $\sd^+$ is the smallest
  $1$-Lipschitz function larger than $d^+_i$ where it is non-negative (Lemma~\ref{lem:defsd}),
  this achieves the proof.
\end{proof}
\subsection{Consistency of the algorithm}\label{sec:consistexplicit}

In this section as before, $\theta\in  {(0,1]}$ is fixed and
the small parameters $\e,h$ are linked through $h=\theta\e^2/ {(2N)}$.
We investigate the limit of the scheme as $\e,h\to 0$.

Let $E^{0,\e},u^{0,\e}$ be defined as in Section \ref{sec:algo_explheat}, so that  $u^{0,\e}$, $1$-Lipschitz,
 with $E^{0,\e} = \{i\in\ez:u^{0,\e}_i < 0\}$ and $\comp{(E^{0,\e})} = \{i\in\ez:u^{0,\e}_i \ge 0\}$.
Note  that  as $\e\to 0$,
 $E^{0,\e}\to  {E^0}$ and $\comp{(E^{0,\e})}\to\comp{(\mathring{E}^0)}$ in Hausdorff distance. If we consider a ``nice'' initial closed set $E^0$, such that $E^0$ is the closure of its interior (and  $\partial E$ has empty interior), and is not empty, nor its complement, we can also choose as $E^{0,\e}$ the set $E\cap \ez$. 
 
We let $\sd^{0,\e}=\sd^+[u^{0,\e}]$. We run the algorithm, obtaining a sequence
$\sd^{k,\e}$ of ``signed distance functions''.

As in the statement of Theorem~\ref{th:mainexplicit}, we let for each $t\ge 0$, $d^\e(t) = \sd^{[t/h],\e}$ and then let,
for $t>0$, $E_\e(t) = \{ i\in\ez: d^\e(t)_i<0\}$.
Then we let $E_\e = \{(i,t)\in\ez\times [0,+\infty): i\in E_\e(t)\}$
and $F_\e = \{(i,t)\in\ez \times [0,+\infty): i\not\in E_\e(t)\}$.
We find a subsequence such that both $E_{\e_k}\to E$
and $F_{\e_k}\to \comp{A}$ in the Kuratowski
sense in $\R^N\times [0,+\infty)$, where $A$ is an open set and
$E$ a closed set. Observe that $A\subset E$.
We let 
\[
    T^*=\inf\{ t>0: E \cap (\R^N\times (t,+\infty))=\emptyset \text{ or }
    A^\complement \cap (\R^N\times (t,+\infty))=\emptyset\},
\]    
with  $T^*\in [0,+\infty]$.
Note that it may (in general will) happen that after some time,
$u^{k+1,\e}$ defined by~\eqref{eq:heat} becomes positive (or negative),
in which case $\sd^{k+1,\e}$ (defined by~\eqref{eq:redistance}) will be $+\infty$
(respectively, $-\infty$) and the corresponding sets $E_\e(t)$ (or $\comp{E_\e(t)}$)
will be empty: in the limit, this corresponds to times which are past the extinction
time $T^*$ of $E$ or $\comp{A}$.

Concerning the distances, we see that
if for some $k$, $\sd^{k,\e}_i\ge R>0$ at $i\in\ez$, then $\sd^{k,\e}_j\ge R-|j-i|$
(as it is $1$-Lipschitz). Therefore, using Lemma \ref{lem:balls} and the fact that $\sd^+[-u]=-\sd^-[u]\ge -\sd^+[u]$, we get 
\[ \sd^{k+1,\e}\ge \sd^+[ -v_{\cdot-i} ] = -\sd^-[v_{\cdot-i}] \ge R-  |\cdot-i|-\frac CR h,  \] 
where $v$ has been defined in \eqref{eq:v}. Therefore,  iterating the argument we  show that for some $C\ge 1$
depending only on the dimension, if $\e$ is small enough,
\[
  \sd^{\ell,\e}_i \ge R- \frac{C}{R}(\ell - k )h
\]
for $\ell \ge k$ and as long as the right-hand side is larger than $R/2$ (that is, $\ell h - k h\le R^2/2C$).

This may also be written as 
\begin{equation}\label{eq:nondecreasing}
  d^\e(s)_i \ge d^\e(t)_i - \frac CR(s-t+h)
\end{equation}
for $d^\e(t)_i\ge R>0$, $0\le t\le s\le CR^2$ for some constant $C$ depending
only on the dimension. A similar, symmetric statement holds if $d^\e(t)_i\le -R<0$.
In particular, by our choice of $E^0$, one has $T^*>0$.

The estimate~\eqref{eq:nondecreasing} allows to reproduce the proof of~\cite[Prop.~4.4]{CMP17}.  {In particular, there exists  a subsequence $\e_k\to 0$ such that,  except for a countable set of times, the function
\[
\sum_{i\in\e_k\Z^N} d^{\e_k}_i(t)\chi_{i+[0,\e_k)^N}
\]
converges locally uniformly to some function $d(x,t)$ which
is locally finite for $t<T^*$. Moreover,  its positive and negative parts
 satisfy (thanks to~\eqref{eq:compdist}, \eqref{eq:compdistm}):}
\begin{equation}\label{eq:distEt}
  d^+(\cdot,t) = \dist(\cdot,E(t))
  \quad\text{ and }\quad
  d^-(\cdot,t) = \dist(\cdot,\comp{A}(t)),
\end{equation}
where $E(t) = \{x\in\R^N: (x,t)\in E\}$ and $A(t) = \{x\in\R^N:(x,t)\in A\}$,
for $t<T^*$. For every $x\in\R^N$, the functions $\dist(x,E(\cdot))$ and $\dist(x,\comp{A}(\cdot))$ are
left-continuous and right-lower-semicontinuous. Equivalently, the maps $E(\cdot)$ and $\comp{A}(\cdot)$ are left-continuous and right-upper-semicontinuous with respect to the
Kuratowski convergence. Finally, $E(0)=E^0$ and $A(0)=\mathring{E^0}$.
In addition, $d(\cdot,t)\equiv +\infty$ or $-\infty$  for all $t>T^*$.

From \eqref{eq:distEt} one has in particular that $\comp{E}\cap (\R^N\times (0,T^*))
= \{(x,t)\in \R^N\times (0,T^*): d(x,t)>0\}$ while
$A  {\cap (\R^N\times (0,T^*))}= \{(x,t)\in \R^N\times (0,T^*): d(x,t)<0\}$, and
\begin{align*}
    d^+(x,t)& = \inf\left\{ \liminf_{k\to+\infty} \max\{0,d^{ {\e_k}}_{i_k}(t_k)\}: \e_k\Z^N\times \R_+\ni (i_k,t_k)\stackrel{k\to\infty}{\longrightarrow} (x,t)\right\}\\
    -d^-(x,t)& = \sup\left\{ \limsup_{k\to+\infty} \min\{0,d^{ {\e_k}}_{i_k}(t_k)\}: \e_k\Z^N\times \R_+\ni (i_k,t_k)\stackrel{k\to\infty}{\longrightarrow} (x,t)\right\}
\end{align*}
are the classical relaxed half-limits (see for instance~\cite{BarSonSou,BarSou}).

There are two elementary directions to prove the convergence
of the algorithm. One can establish the consistency with the
viscosity approach of~\cite{Soner93}, showing that $d$ is a viscosity
super-solution of the heat equation in $\{d>0\}$ (and a subsolution
in $\{d<0\}$), or equivalently the
consistency with respect to the distributional
Definition~\ref{def:distributionalsuperflow}.
We describe both proofs, as the first one is more natural
from a historical point of view, yet we found  easier to generalize
the second to non-compactly supported kernels (see Section~\ref{sec:kernels}).

\subsubsection*{The viscosity approach}
For the viscosity point of view, the idea is to consider a smooth
test function $\eta(x,t)$ with $\eta\le d$, $\eta(\bar x,\bar t)=d(\bar x,\bar t)>0$, and assume without loss of generality that the contact point is unique~\cite{CIL}.
Then, it is standard that for small $\e_k$, there is
$i_k\to \bar x$, $i_k\in\e_k\Z^N$, and $t_k\to \bar t$ such that
for all $t>0$ and $i\in \e_k\Z^N$:
\[
  \eta_k(i,t)=  \eta(i,t) + (d^{\e_k}(t_k)_{i_k} - \eta(i_k,t_k) )
  \le d^{\e_k}(t)_i,
  \quad \eta_k(i_k,t_k) = d^{\e_k}(t_k)_{i_k} >0.
\]

We have:
\begin{multline*}
  \eta_k(i_k,t_k) =d^{\e_k}(t_k)_{i_k}\\
  \ge
(1-\theta) d_{i_k}^{\e_k}(t_k-h_k) + \frac{\theta}{2N}\sum_{n=1}^N
\left(d_{i_k+\e_k e_n}^{\e_k}(t_k-h_k)
  + d_{i_k-\e_k e_n}^{\e_k}(t_k-h_k)\right)
\\ 
\ge (1-\theta) \eta_k(i_k,t_k-h_k) + \frac{\theta}{2N}\sum_{n=1}^N
  \left(\eta_k(i_k+\e_k e_n,t_k-h_k) + \eta_k(i_k-\e_k e_n,t_k-h_k)\right),
\end{multline*}
so that:
\[
  \frac{\eta_k(i_k,t_k)-\eta_k(i_k,t_k-h_k)}{h_k}
  \ge \left(\Delta_{\e_k}\eta_k(\cdot,t_k-h_k)\right)_{i_k}.
\]
Using that $\eta$ is smooth and passing to the limit, we recover:
\[
  \frac{\partial\eta}{\partial t}(\bar x,\bar t)\ge \Delta\eta(\bar x,\bar t),
\]
showing that $d^+$ is a viscosity supersolution of the heat equation
in $\{d>0\}$.
\subsubsection*{The distributional approach}
On the other hand, the variational point of view is tackled as follows,
considering rather a test function
 $\eta\in C_c^\infty(\comp{E}\cap(\R^N\times (0,T^*));\R_+)$. The support $U_\eta$  {of $\eta$
is compact and at positive} distance from $E$, hence for $\e_k$ small enough it
is also at positive distance from $E_{\e_k}$ so that
$d^{\e_k}$ is bounded from below by a positive number on $\overline{U}_\eta\cap (\ez\times [0,+\infty))$.
Thanks to~\eqref{eq:subsoldiscrete}, it follows that
\[
\frac{d^{\e_k}(t)_i - d^{\e_k}(t-h_k)_i}{h_k} \ge (\Delta_{\e_k} d^{\e_k}(t-h_k))_i
\]
for $(i,t)\in \overline{U}_\eta$, hence:
\[
  \int_0^{T^*} {\e_k^N}\sum_{i\in {\e_k}\Z} \left(\frac{d^{\e_k}(t)_i - d^{\e_k}(t-h_k)_i}{h_k} - (\Delta_{\e_k} d^{\e_k}(t-h_k))_i\right)\eta(i,t) dt \ge 0.
\]
Rearranging the sums, this is also:
\[
  \int_0^{T^*} {\e_k^N}\sum_{i\in {\e_k}\Z} \left(\frac{\eta(i,t) - \eta(i,t+h_k)}{h_k} - (\Delta_{\e_k} (\eta(t+h_k,\cdot))_i\right)d^{\e_k}(t)_i dt \ge 0.
\]
In the limit (since $\eta$ is smooth, and $d^{\e_k}$ converges uniformly for almost each time), we obtain
\[
  \int_{0}^{T^*}\int_{\R^N} (-\partial_t\eta  - \Delta \eta )d\,dxdt
\ge 0
\]
so that
\begin{equation}\label{eq:distsupsol}
  \frac{\partial d}{\partial t} \ge \Delta d \quad\text{ in }\quad
  \mathcal{D}'(\{(x,t)\in\R^N\times [0,T^*): d(x,t)>0\}),
\end{equation}
that is in the sense of distributions (or measures). In the same way, we have:
\begin{equation}\label{eq:distsubsol}
  \frac{\partial d}{\partial t} \le \Delta d \quad\text{ in }\quad
  \mathcal{D}'(\{(x,t)\in\R^N\times [0,T^*): d(x,t)<0\}).
\end{equation}

\begin{figure}[htb]
\includegraphics[width=0.38\textwidth]{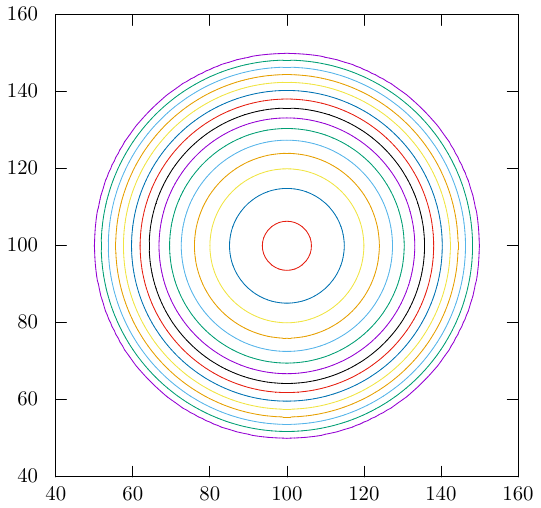}\hfill
\includegraphics[width=0.51\textwidth]{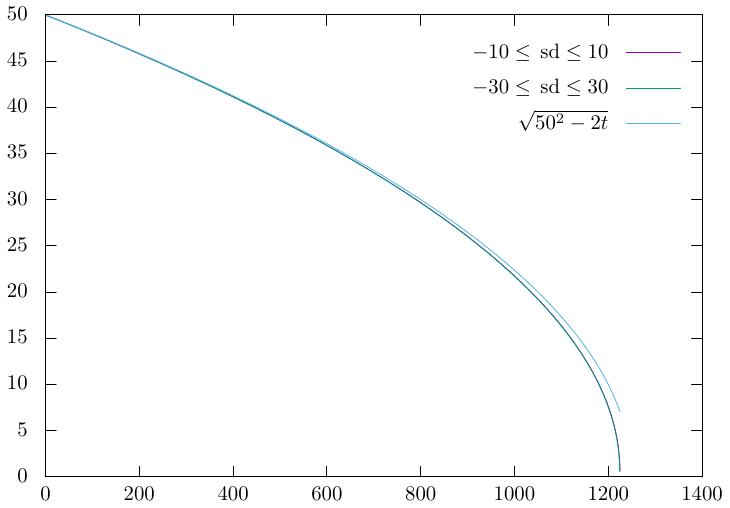}
\caption{Evolution of a disk (initial radius 50) (left) and decay of the radius (right)}\label{fig:Explicit}
\end{figure}

\subsection{Examples} We illustrate
with some numerical experiments the result in Theorem~\ref{th:mainexplicit}.
In Figure~\ref{fig:Explicit}, we plot the evolution of a circle of radius $50$ pixels at equally spaced times
(the computation assumes $\e=1$ and $h=\e/(2N)=0.25$), until extinction at 
$t= 1250$ (actually, extinction occurs too early,
at iteration $4898$, that is time $t=1224.5$). 

We also plot the decay of the radius. We
see that it perfectly follows the exact rate, except at
too small
scales where the numerical evolution is a bit too fast. 
The distance is evaluated with the inf-convolution formulae~\eqref{eq:sdp}, truncated at a certain level $\bar d$
to reduce the cost of this step.
We tried different levels of truncation, yet for this particular
experiments, did not see any gain at truncating at levels
larger than $\bar d=10$: the plot is for $\bar d=30$, and the decay of
the radii is plotted for $\bar d\in \{10,30\}$ yet the two
curves are exactly superimposed.

\begin{figure}[htb]
\includegraphics[width=0.45\textwidth]{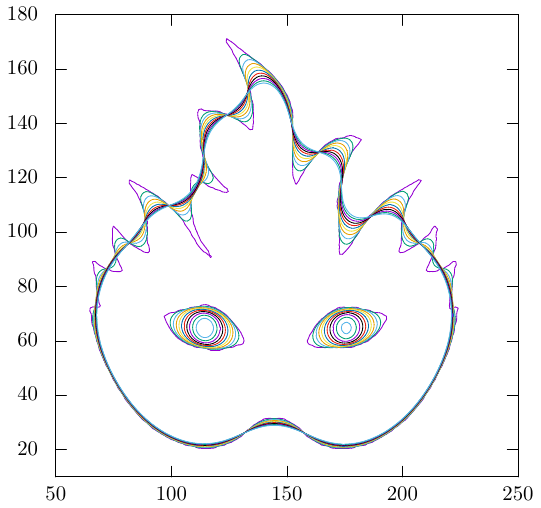}\hfill
\includegraphics[width=0.45\textwidth]{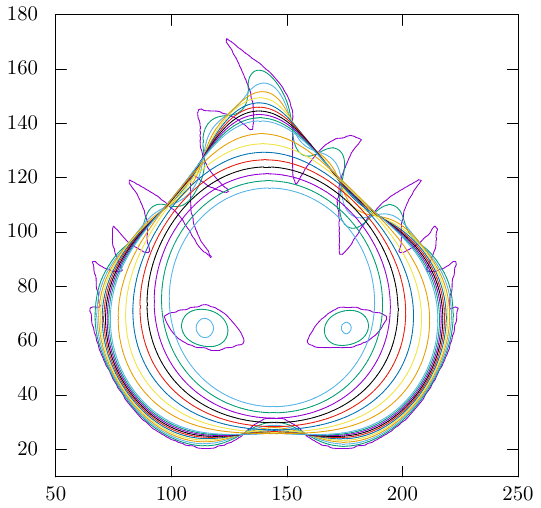}
\caption{Evolution of a shape: left at times $t=0,20,\dots,200$,
right at times $t=0,25,50,\dots,250$ and then $t=375,500,\dots,1250$.}\label{fig:ExplicitMask}
\end{figure}
Figure~\eqref{fig:ExplicitMask} shows the evolution starting from a more complex initial shape, at different successive times.

\section{Extension: general convolution kernels}
\label{sec:kernels}
\subsection{Convolution kernels}
Our idea is now to generalize the algorithm in~\eqref{eq:heat} as follows: given
$\sd_i^{0,\e}$ defined upon an initial set $E^0$, we define for all $k$:
\begin{equation}\label{eq:defalgo}
  \begin{aligned}
    & u^{k+1,\e} = K^\e*\sd^{k,\e}\\
    & \sd^{k+1,\e} = \sd[u^{k+1,\e}]
  \end{aligned}
\end{equation}
where the discrete convolution $*$ is defined for $u,v:\ez\to\R$ as 
\[ 
(u * v)_i = \sum_{j\in\ez} u_j v_{i-j},
\]
$K^\e$ is a convolution kernel and
$\sd[\cdot]$ is one of the redistancing operations defined in~\eqref{eq:sdp} or~\eqref{eq:sdm}.
In the next Section~\ref{sec:Kernel}
we introduce the properties of the kernels $K^\e$
for which convergence is expected. 
Our main result is the following.
\begin{theorem}\label{th:variantexplicit}
  We assume that $E^0$ is such that the generalized mean curvature flow $E(t)$ starting from $E^0$
  is uniquely defined and denote $d_{E(t)}$ the signed distance function to $\partial E(t)$.
  
  Assume that $K^\e$ satisfies~\eqref{Kpositive}, \eqref{Kone}, \eqref{Kh} and~\eqref{eq:approxsymbol} below. 
  Then provided $h\gtrsim \e^2$ as $  {\e,h}\to 0$,
  \[
    \sum_{i\in\e\Z^N}\sd^{[t/h],\e}_i\chi_{i+[0,\e)^N} \to d_{E(t)}
  \]
  locally uniformly in $\R^N$ for all $t\ge 0$ but a countable number.
\end{theorem}
The proof of this result, based on an adaption of the proof of Theorem~\ref{th:mainexplicit} will
be developed in Sections~\ref{sec:balls2}, \ref{sec:proof}. 

\subsection{Kernel properties} \label{sec:Kernel}
For any $\e>0$ we consider  kernels $K_j^\e$, $j\in\e\Z^N$, with:
\begin{align}
  & \textstyle  K_j^\e= K_{-j}^\e \ge 0 \quad \text{ for all } j\in\e\Z^N, \label{Kpositive} \\
   & \textstyle \sum_{j\in\e\Z^N} K_j^\e=1, \label{Kone}\\
  & \textstyle h=h(\e):= \frac{1}{2N}\sum_{j\in\e\Z^N} |j|^2K_j^\e \to 0 \quad \text{  as } \quad  \e\to 0.  \label{Kh}  
\end{align}

We introduce the discrete Fourier transform, for $\xi\in\R^N$:
\[
  \widehat{K^\e}(\xi) := \e^N\sum_{j\in\e\Z^N} K^\e_j e^{-2i\pi j\cdot\xi}\,,\quad  \widetilde{K^\e} :=\e^{-N}\widehat{K^\e}
\]
which is a $1/\e$-periodic function with $|\widetilde{K^\e}|\le 1= \widetilde{K^\e}(0)$, and we
assume in addition that:
\begin{equation}\label{eq:approxsymbol}
  \lim_{\e\to 0} \frac{1-\widetilde{K^\e}(\xi)}{h} = 4\pi^2 |\xi|^2
\end{equation}
for any $\xi\in\R^N$.

We need \eqref{eq:approxsymbol}  in order to ensure that   
$(K^\e*\eta-\eta)/h$ goes to $\Delta \eta$
as $\e\to 0$ for any test function $\eta$. In particular, the reason for the definition of
$h$ in~\eqref{Kh} is precisely that for $\eta=|x|^2$, $\Delta\eta = 2N$. We observe
that~\eqref{Kh} is equivalent to:
\[
  h =h(\e):= -\frac{\Delta \widetilde{K^\e}(0)}{8\pi^2N} \to 0\quad \text{as}\quad \e\to 0.
  \leqno (\ref{Kh}')
\]

\begin{remark}\label{rmk:expl_lapl}
    We remark that the Euler's explicit  algorithm  considered in the previous Section~\ref{sec:consistexplicit} falls in the present framework.   We fix a time step $\tau(\e)$ converging to $0$ as $\e\to 0^+$ and we start by observing  that \eqref{eq:heat} (with $\tau$ in place of $h$) can be rewritten as in the first equation of \eqref{eq:defalgo}, with the kernel $K^\e$ given by 
\[    
K^\e=\big(1-\tfrac{2N\tau}{\e^2}\big)\delta_0 + \frac{\tau}{\e^2}\sum_{k=1}^N (\delta_{\e e_k}+\delta_{-\e e_k}).
\]
 Here $(\delta_j)_i=1$ if $i=j$ and $0$ otherwise and we require $0<\tau\leq \frac{\e^2}{2N}$. 
  Note that $K^\e\ge 0$, it is symmetric, $\sum_{j\in\ez} K^\e_j=1$ and 
    \begin{equation}\label{eq:defhexpl}
        h:=\frac1{2N}\sum_{j\in\ez} |j|^2 K^\e_j=\frac{\tau}{N\e^2}\sum_{k=1}^N |\e e_k|^2=\tau,    
    \end{equation}
    so that, in particular,  \eqref{Kh} is satisfied. Lastly, let us check \eqref{eq:approxsymbol}. Since
    \[
    \widetilde{K^\e}(\xi)=\sum_{j\in\ez} K^\e_j e^{-2\pi i j\cdot \xi}= \sum_{j\in\ez}  K^\e_j \cos(2\pi j\cdot \xi)=1 + \frac{2\tau}{\e^2}\sum_{k=1}^N(\cos(2\pi \e e_k\cdot \xi)-1),
    \]
    using also \eqref{eq:defhexpl}, we deduce
    \[
    \frac{1-\widetilde{K^\e}(\xi)}h = \frac2{\e^2} \sum_{k=1}^N(1-\cos(2\pi \e e_k\cdot \xi))\to 4\pi^2 |\xi|^2
    \]
    as $\e\to0$.
\end{remark}
 
\begin{remark}\label{rmk:symmetries}
  {Assume that a kernel $K^\e$ satisfies \eqref{Kpositive}-\eqref{Kh}.} Since $\widetilde{K^\e}$ has a maximum at $0$, $D\widetilde{K^\e}(0)=0$,
$D^2\widetilde{K^\e}(0)\le 0$ and we have:
\begin{multline} \label{eq:KD2K}
  \frac{\widetilde{K^\e}(\xi)-1}{h} = \left(\frac{1}{h}\int_0^1 (1-s)D^2\widetilde{K^\e}(s\xi)ds\right)\xi\cdot\xi
  \\ 
  = \frac{1}{2h}\left(D^2\widetilde{K^\e}(0)\xi\right)\cdot\xi + \left(\frac{1}{h}\int_0^1 (1-s)(D^2\widetilde{K^\e}(s\xi)-D^2\widetilde{K^\e}(0))ds\right)\xi\cdot\xi.
\end{multline}
In particular, by~$\eqref{Kh}'$ it follows that\footnote{It is easy to check using the definition that $D^2\widetilde{K^\e}(\xi)\ge D^2\widetilde{K^\e}(0)$.}
\begin{equation}\label{eq:boundsymbol}
  0\le \frac{1-\widetilde{K^\e}(\xi)}{h}\le 4\pi^2 N |\xi|^2
\end{equation}
for any $\xi\in\R^N$. In addition, in case the kernels satisfy also that for all $\xi\in\R^N$,
  $D^2\widetilde{K^\e}(\xi)-D^2\widetilde{K^\e}(0)=o(h)$
  as $\e\to 0$, and $\frac{1}{2h}D^2\widetilde{K^\e}(0)\to -A$ where $A$ is some positive semi-definite
  symmetric matrix, we find
  \[
    \frac{\widetilde{K^\e}(\xi)-1}{h} \to - (A\xi)\cdot\xi
  \]
  as $\e\to 0$ for any $\xi\in\R^N$; eq.~\eqref{eq:approxsymbol} corresponds to the case $A=I_N$.
  
     {In particular, if \eqref{Kpositive}-\eqref{Kh} hold, and if $K^\e$ is symmetric under coordinate permutations and reflections in each coordinate, then:}
\[
  \sum_{j\in\e\Z^N} j\otimes j K^\e_j = \left(\frac{1}{N}\sum_{j\in\e\Z^N} |j|^2K^\e_j  \right)I_N,
\]
which can be written as:
\begin{equation}\label{eq:symK}
  \frac{1}{2h} D^2 \widetilde{K^\e}(0) = -4\pi^2 I_N
\end{equation}
for all $\e>0$. Then \eqref{eq:KD2K} yields:
   \[
   \frac{\widetilde{K^\e}(\xi)-1}{h}  
    = -4\pi^2|\xi|^2 + \left(\frac{1}{h}\int_0^1 (1-s)(D^2\widetilde{K^\e}(s\xi)-D^2\widetilde{K^\e}(0))ds\right)\xi\cdot\xi
  \]
  Therefore, in case the kernels satisfy also that for all $\xi\in\R^N$,
  $D^2\widetilde{K^\e}(\xi)-D^2\widetilde{K^\e}(0)=o(h)$
  as $\e\to 0$, then~\eqref{eq:approxsymbol} follows.   {This can be verified by an explicit computation for the kernel $K^\e$ associated with Euler's explicit scheme (\textit{cf} Remark~\ref{rmk:expl_lapl}).}
\end{remark}

\newcommand{\Ksd}{\mathcal{K}^\e}
\subsection{Control of the balls}\label{sec:balls2}
To obtain  convergence of the algorithm for  general kernels $K^\e$, we first need to show
an equivalent of the estimate in Lemma~\ref{lem:balls} (to obtain
compactness in time of the distances).
We first show the following estimate.
\begin{lemma}\label{lem:ballsKh}
  Let $K^\e$ satisfy~\eqref{Kpositive}, \eqref{Kone}, \eqref{Kh}, and let $R>0$.
  Then for any $i\in\e\Z^N$ with $|i|\ge \frac {3R}4$ it holds 
  \[
    (K^\e*|\cdot|)_i \le |i| + \frac{  {6}N}{R}h
  \]
\end{lemma}
\begin{proof}
  We write:
  \begin{multline*}
    \sum_{j\in\e\Z^N} K_j|i-j|  \le \sum_{|j|\ge R/2} K_j (|i|+|j|)+  \sum_{|j|< R/2} K_j |i-j|
    \\  = |i| +  \sum_{|j|\ge R/2} K_j|j| + \frac{1}{2}  \sum_{|j|< R/2} K_j (|i-j|+|i+j|-2|i|).
  \end{multline*}
  We have, using~\eqref{Kh}:
  \[
    \sum_{|j|\ge R/2} K_j|j|\le \frac{2}{R} \sum_{|j|\ge R/2} K_j|j|^2 \le \frac{4Nh}{R}.
  \]
  For $|i|\ge 3R/4$, $|j|\le R/2$ we also have:
  \[
    |i-j|+|i+j| \le 2|i|  + \frac{2}{R}|j|^2
  \]
  (reasoning as in \eqref{eq:Taylor} in  the proof of Lemma~\ref{lem:balls}), so that
  \[
    \sum_{|j|< R/2} K_j (|i-j|+|i+j|-2|i|) \le \frac{2}{R}\sum_{|j|<R/2} K_j|j|^2 \le \frac{4Nh}{R}
  \]
using again~\eqref{Kh}.  The result follows.
\end{proof}

\begin{corollary}\label{lem:balls2}
  Assume $h(\e)\ge \e^2$. Then there
  exists $C>0$ such that for any $R>0$, there is $\e_0>0$ such that if $\e\le\e_0$,
  \[
    \sd[K^\e*|\cdot|-R]_i \le |i|-R +\frac{C}{R} h\quad \forall i\in\e\Z^N. 
  \]
\end{corollary}
This results directly from Lemma~\ref{lem:ballsKh} and the estimate
on the redistancing established in the proof of Lemma~\ref{lem:balls}.

\subsection{Consistency} \label{sec:proof}
We will show the following convergence result,
which has Theorem~\ref{th:variantexplicit} as a corollary.
\begin{theorem}\label{th:mainimplicit}
  Let $E^0\subseteq \R^N$   {be closed} and consider the algorithm \eqref{eq:defalgo}. As $\e\to 0$, the function $d^\e$, defined for $t\ge 0$ and $i\in\e\Z^N$
  by $d^\e(t)_i = \sd^{[t/h],\e}_i$, converge up to subsequences,
  for almost all times and locally uniformly in space to a function
  $d(x,t)$ such that $d^+=\max\{d,0\}$ is the distance function
  to a supersolution of the (generalized)
  mean curvature flow starting from $E^0$,
  and $d^-$ is the distance function to a supersolution starting from
  $\comp{E^0}$. In particular, if the mean curvature flow $E(t)$ starting
  from $E^0$ is unique, then $d^\e(t)$ converges to the signed distance
  function to $E(t)$, up to extinction.
\end{theorem}

Again, we expect that with our assumptions, it should not be
too difficult to show consistency in the viscosity sense. This requires
to show a (locally continuous) convergence of $\frac{K_\e-\delta}{h}*\eta_{|\e\Z^N}$
to $\Delta\eta$ as $\e\to 0$, for smooth tests $\eta$.
It seems however that proving consistency with the distributional solutions of Definition~\ref{def:distributionalsuperflow}
is a bit more direct.
\begin{proof}
  {For $k\in\N$ let $j\in\ez$ with $\sd^{k+1}_j>0$. Starting from    \eqref{eq:defalgo} and  since $\sd^{k+1,}_j\ge u_j$  by Lemma \ref{lem:defsd}, we deduce:}
\[
  \frac{\sd^{k+1,\e}_j-\sd^{k,\e}_j}{h} \ge \left(\frac{K^\e- \delta}{h}*\sd^{k,\e}\right)_j,
\]
%in any $j$ with $\sd^{k+1}_j>0$, 
where $\delta = (\delta_j)_{j\in\e\Z^N}$ is defined by $\delta_0=1$,
$\delta_j=0$ for $j\neq 0$. We write this inequality as:
\[
  \frac{d^{\e}(t+h)-d^{\e}(t)}{h} \ge \frac{K^\e- \delta}{h}*d^{\e}(t)
\]
for any $t\ge 0$ where as before, $d^\e(t) = \sd^{[t/h],\e}$.
Thanks to the estimate in Corollary~\ref{lem:balls2}, we obtain,
as in Section \ref{sec:consistexplicit} (see \eqref{eq:nondecreasing}), the same compactness
as in~\cite[Prop.~4.4]{CMP17}. Therefore, up to a subsequence,  $d^\e$ converges
locally uniformly in space and for all times but a countable number (and until an extinction time $T^*\in (0,+\infty]$), to
$d(x,t)=\dist(x,E(t))-\dist(x,\comp{A}(t))$, where $E$ is the Kuratowski limit of $\{(j,t): d^\e(t)_j\le 0\}$
and ${A}\subset E$ the complement of the Kuratowski limit of $\{(j,t): d^\e(t)_j\ge 0\}$. We will show that $E$ is a superflow,
${A}$ a subflow, and in case $E^0$ does not develop ``fattening'' (that is, the mean
curvature evolution from $E^0$ is unique), $A = \mathring{E}$, $E=\overline{A}$ and
$d(x,t)=\sd_{E(t)}(x)$.

By symmetry of the statements,
it is enough to show that $E$ is a superflow. Let
$\eta\in C_c^\infty(\comp{E}\cap (\R^N\times (0,T^*));\R_+)$, then for $\e$ small enough, since $\eta>0$ implies $ d^\e>0$, we can write:
\[
\begin{split}
    &\int_{\R_+}\e^N\sum_{j\in\e\Z^N} \eta(j,t)  \frac{d^{\e}_j(t+h)-d^\e_j(t)}{h} \,dt \\
    &\ge
  \int_{\R_+}\e^N\sum_{j\in\e\Z^N} \sum_{j'\in \e\Z^N} \eta(j,t) \frac{K^\e_{j-j'}-\delta_{j-j'}}{h}d^\e_{j'}(t)
  \,dt
\end{split}
\]
The convergence of  the left hand side does not raise any difficulty:
\[
\lim_{\e\to 0}  \int_{\R_+}\e^N\sum_{j\in\e\Z^N} \eta(j,t)  \frac{d^{\e}_j(t+h)-d^\e_j(t)}{h} \,dt =  -\int_{\R_+}\int_{\R^N} \frac{\partial\eta}{\partial t} d \,dxdt,
\] 
and we want to show convergence also for the
right hand side:
\begin{equation}\label{eq:weakconvlap}
\lim_{\e\to 0}  \int_{\R_+}\e^N\sum_{j\in\e\Z^N} \sum_{j'\in \e\Z^N} \eta(j,t) \frac{K^\e_{j-j'}-\delta_{j-j'}}{h}d^\e_{j'}(t)
  \,dt
=  \int_{\R_+}\int_{\R^N} d\Delta\eta \,dxdt.
\end{equation}
We note that 
there exists $C>0$ such that, for all $t$ such that the support of $\eta(\cdot,t)$ is
not empty,
$|d_j^\e(t)|\le C(1+|j|)$. 
Let $  {R_0\ge 1}$ such that $\text{spt}(\eta)\subset\subset B(0,R_0)\times (0,T^*)$. If $\e>0$ is
small enough then for any $R>2R_0$ and $t<T^*$,
\begin{align}\label{eq:boundremainder}
  &\e^N\left|\sum_{j':|j'|>R} \sum_{j\in \e\Z^N} \eta(j,t) \frac{K^\e_{j-j'}-\delta_{j-j'}}{h}d^\e_{j'}(t)\right|
  \nonumber \\
  &=
  \e^N\left|\sum_{j: |j|<R_0}  \eta(j,t)\sum_{k:|j-k|>R}\frac{K^\e_{k}-\delta_{k}}{h}d^\e_{j-k}(t)\right|
  \\ 
  &\le |B(0,R_0)| \lVert\eta\rVert_\infty \sum_{k: |k|>R-R_0} \frac{|k|^2 K^\e_k}{h} \frac{C(1+|k|+R_0)}{|k|^2}
  \le 2N\frac{3C|B(0,R_0)|\lVert\eta\rVert_\infty}{R-R_0}\to 0\nonumber
\end{align}
as $R\to\infty$ (uniformly in $\e$), using~\eqref{Kh} for the   {second} inequality.

We fix $t\in (0,T^*)$ and we now show that for any $\varphi\in L^1(\R^N)$,
\begin{equation}\label{eq:weakconv}
  \e^N \sum_{j\in\e\Z^N} \sum_{j'\in\e \Z^N} \eta(j,t)  \frac{K^\e_{j-j'}-\delta_{j-j'}}{h}\varphi^\e_{j'}
  \to \int_{\R^N} \Delta \eta (x,t)\varphi(x) \,dx
\end{equation}
as $\e\to 0$, where $\varphi^\e_{j'} = \e^{-N}\int_{j'+[0,\e)^N} \varphi\,dx$ for any $j'$.
Denoting 
\[
u^\e_{j'}= \sum_{j\in\e\Z^N}  \eta(j,t)  \frac{K^\e_{j-j'}-\delta_{j-j'}}{h},
\]
we observe that $u^\e\in \ell^1(\ez)\cap \ell^\infty(\ez),\varphi^\e\in \ell^1(\e\Z^N)$, and that
Plancherel's formula holds:
\[
  \e^N \sum_{j'\in\e\Z^N} u^\e_{j'} \varphi^\e_{j'} = \int_{[-\frac{1}{2\e},\frac{1}{2\e}]^N}
  \widehat{u^\e}(\xi) \overline{\widehat{\varphi^\e}}(\xi) d\xi.
\]
This follows from the calculation
\[
   \int_{  {[-\frac{1}{2\e},\frac{1}{2\e}]^N}}
   \widehat{u^\e}(\xi) \overline{\widehat{\varphi^\e}}(\xi) d\xi
   = \e^{2N} \sum_{j,j'}\int_{[-\frac{1}{2\e},\frac{1}{2\e}]^N} u^\e_j e^{-2i\pi j\cdot \xi} \varphi^\e_{j'}e^{2i\pi j'\cdot\xi}\,d\xi
   = \e^N\sum_j u^\e_j \varphi^\e_j
 \]
 which obviously holds for finite sums, and then in the limit. 
 We denote $(\eta^\e_j)_{j\in\e\Z^N}$ the restriction of $\eta(\cdot,t)$ to  $\e\Z^N$.
 Then, we have:
 \[
   \widehat{u^\e}(\xi) = \e^N\sum_{j'\in\e\Z^N}\sum_{j\in\e\Z^N} \eta(j,t)\frac{K^\e_{j-j'}-\delta_{j-j'}}{h}
   e^{-2 i\pi (j+(j'-j))\cdot\xi}
   = \widehat{\eta^\e}(\xi) \frac{\widetilde{K^\e}(\xi)-1}{h}
 \]
 where as before $\widetilde{K^\e} = \e^{-N}\widehat{K^\e}$.
 In particular thanks to~\eqref{eq:boundsymbol}, $|\widehat{u^\e}(\xi)|\le 4\pi^2 N|\xi|^2|\widehat{\eta^\e}(\xi)|$, and  thanks to~\eqref{eq:approxsymbol},
 $\widehat{u^\e}(\xi)\to -4\pi^2|\xi|^2\widehat{\eta}(\xi,t)$ as $\e\to 0$ for all $\xi\in\R^N$,
where $\widehat{\eta}(\cdot,t)$ denotes the Fourier transform in $\R^N$ of $\eta(\cdot,t)$.
 Lemma~\ref{lem:uniformSchwartz} then shows that $\widehat{u^\e}$, extended by
 $0$ out of $[-\frac{1}{2\e},\frac{1}{2\e}]^N$, is rapidly decaying with estimates
 independent of $\e$. We note that the constants in these estimates depend
 only on the regularity of $\eta(\cdot,t)$ and can also be chosen to not depend on $t$.
In particular there exists $M>0$ independent of $t$ such that for all $\xi\in[-\frac{1}{2\e},\frac{1}{2\e}]^N$,
 \[
   |\widehat{u^\e}(\xi)|\le \frac{M}{1+|\xi|^{N+1}},
 \]
 which shows that
 \begin{equation}\label{eq:boundetaK}
\left|\sum_{j'\in\e\Z^N}  \eta(j',t)  \frac{K^\e_{j'-j}-\delta_{j'-j}}{h}\right|
=
   |u_j^\e| = \left|\int_{[-\frac{1}{2\e},\frac{1}{2\e}]^N} \widehat{u^\e}(\xi)e^{2\pi i \xi\dot j}d\xi
   \right|\le CM
   \end{equation}
   for some dimensional constant $C>0$.

 On the other hand, the functions $\widehat{\varphi^\e}$ are such that for any $\xi\in\R^N$,
  \[
    |\widehat{\varphi^\e}(\xi)|\le \lVert\varphi\rVert_{L^1(\R^N)},\quad
    \widehat{\varphi^\e}(\xi) \to \widehat{\varphi}(\xi).
  \]
  From Lebesgue's dominated convergence theorem we deduce:
   \[
   \int_{[-\frac{1}{2\e},\frac{1}{2\e}]^N}
   \widehat{u^\e}(\xi) \overline{\widehat{\varphi^\e}} d\xi
   \stackrel{\e\to 0}{\longrightarrow}
   \int_{\R^N}  -4\pi^2|\xi|^2\widehat{\eta}(\xi,t)\overline{\widehat{\varphi}}\,d\xi
   = \int_{\R^N} \Delta\eta(x,t)\varphi(x)\,dx.
 \]
 This shows~\eqref{eq:weakconv}.

 We now fix $\delta>0$ and consider a spatial cut-off $\psi\in C_c^\infty(\R^N;[0,1])$ with $\psi\equiv 1$ on  $B(0,R)$, where (using~\eqref{eq:boundremainder}) $R$ is such that, for $\e>0$
 small enough,
\begin{equation}\label{eq:est_1-psi}
\left|  \int_{\R_+}\e^N\sum_{j':|j'|>R} \sum_{j\in \e\Z^N} \eta(j,t) \frac{K^\e_{j-j'}-\delta_{j-j'}}{h}d^\e_{j'}(t)\,dt\right|
  \le \delta,
\end{equation}
and $R>R_0$ with spt$(\eta)\subset\joinrel\subset B(0,R_0)\times (0,T^*)$.
We let $\varphi(x,t):=\psi(x)d(x,t)$ and as before $\varphi^\e_j(t) 
= \e^{-N}\int_{j+[0,\e)^N} \varphi(x,t)\,dx$ for any $j$. Using~\eqref{eq:weakconv} for
all $t\in (0,T^*)$ and the uniform bound~\eqref{eq:boundetaK}, we obtain
\begin{equation}\label{eq:est_varphi}
\lim_{\e\to 0}\int_0^{T^*} \e^N \sum_{j\in\e\Z^N} \sum_{j'\in\e \Z^N} \eta(j,t)  \frac{K^\e_{j-j'}-\delta_{j-j'}}{h}\varphi^\e_{j'}(t)\, dt
= \int_0^{T^*} \int_{\R^N} \Delta \eta (x,t)d(x,t) \,dx  dt
\end{equation}
thanks to Lebesgue's theorem, and using that $\varphi=d$ in the support of $\eta$.

By the locally uniform convergence of $d^\e$ to $d$
and the continuity of $\varphi$ one also has
that $\varphi^\e_j(t)-d_j^\e(t)\psi(j)\to 0$ uniformly as $\e \to 0$,
and vanishes at distance $\sqrt{N}\e$ of the support of $\psi$.
Hence,  splitting
$d^\e_{j'}(t)=\varphi^\e_{j'}(t) + (\psi(j')d_{j'}^\e(t)-\varphi^\e_{j'}(t)) + (1-\psi(j'))d_{j'}^\e(t)$ and using~\eqref{eq:boundetaK}, \eqref{eq:est_1-psi} and \eqref{eq:est_varphi}, one obtains
\begin{equation*}
  \limsup_{\e\to 0}
  \left| \int\e^N\!\!\sum_{j\in\e\Z^N} \sum_{j'\in \e\Z^N}\!\! \eta(j,t) \frac{K^\e_{j-j'}-\delta_{j-j'}}{h}d^\e_{j'}(t)
   \,dt -
   \iint \Delta \eta (x,t)d(x,t) \,dx  dt\right|
 \le \delta
\end{equation*}
and since $\delta$ is arbitrary we deduce that
\[  -\int_0^\infty \int_{R^N} \frac{\partial\eta}{\partial t} d\,dxdt
  \ge \int_0^\infty \int_{R^N} d\Delta\eta dxdt
\]
that is,
\begin{equation}\label{eq:super}
  \partial_t d \ge \Delta d
\end{equation}
in $\comp{E}$ in the distributional sense.
\end{proof}

\section{Nonlinear redistancing: applications to  deep learning based approximation algorithms }\label{sec:nonlinear}

In this section we give
a partial explanation to the good numerical results produced in~\cite{DeepLearningMCF}
by fully learned approaches for the mean
curvature flow.
There, the authors propose to learn a Neural Network
whose structure mimicks a convolution-redistancing scheme.
Starting from a discrete 2D or 3D image,
supposed
to approximate a phasefield 
\[
u(x,t)= \gamma\left(\frac{1}{\sigma}\sd_E(x,t)\right),
\]
where $\g$ is the Modica-Mortola optimal profile and $\sigma>0$, their simplest network is implementing a convolutional filter followed by a nonlinear operation.
The convolutional filter 
approximates $\rho*u$ for some (learned) nonnegative kernel $\rho$, while the  nonlinear operation   is supposed to rebuild
from the convolution a profile $V\approx\gamma(\frac{1}{\sigma}\sd_E(x,t+h))$
after a time-step $h>0$.
Except for the non-linearity, which
is a pointwise update, this is quite close to the
process described in this paper.
Our approach seems to justify that learning the filter
is easy and robust: indeed, we have seen that basically
any symmetric enough filter is able to approximate the
mean curvature flow, with an appropriate redistancing (see Remark~\ref{rmk:symmetries}).
In addition, it is not difficult to check that this
flow is also characterized by the property that
$v$ defined above is a super/sub solution of
the heat flow in/out the evolving shape, see below for details. It means that one
could imagine to replace the signed distance function with such a non-linear profile in the scheme. 
In this section, we show how to make this heuristic rigorous.

\subsection{Algorithm and control of the balls}
We denote $\gamma$  the optimal profile of some Cahn-Hilliard (a.k.a.~Modica-Mortola) energy
\[
\int \frac\sigma2 |\nabla u|^2dx + \frac{1}{\sigma}\int W(u)dx
\]
where $0<\sigma\in\R$ and  $W$ is a (smooth) two-wells potential, with $W(t)>W(-1)=W(1)$ for $t\not\in\{-1,1\}$. This means
that $\gamma: \R\to (-1,1)$ is the minimizer of
\[
\int_{-\infty}^{+\infty} \frac 12  (\gamma'(s))^2 + W(\gamma(s)) ds
\]
with $\gamma(0)=0$, $\lim_{s\to\pm \infty}\gamma(s)=\pm 1$, which satisfies both $\gamma''(s) = W'(\gamma(s))$ and $ \gamma'(s) =\sqrt{2W(\gamma(s))}$ 
for $s\in\R$,
see for instance~\cite{Alberti96}.  For
simplicity, we use the standard choice $W(t)=(1-t^2)^2/2$. In this instance, one finds that 
\[
\g(s)=\tanh(s),
\]
  {in particular $|\gamma|\le 1$,} and that it holds
\[
\g'(s) = 1-\g^2(s)\ge 0, \quad \g''(s)=-2\g(s)\g'(s).
\] 
We   consider a kernel $K^\e$ satisfying~\eqref{Kpositive}, \eqref{Kone}, \eqref{Kh} and define 
the following algorithm. Given
$u^{0,\e}:=\sd^{0,\e}$ defined upon an initial closed   {and non empty} set $E^0\subseteq \R^N$, we define for all $k\in\N$
\begin{equation}\label{eq:alg_sdg}
    u^{k+1,\e}=\sd[\gamma^{-1}(K^\e\ast \gamma(u^{k,\e}))].
\end{equation}
Let us note that the algorithm above is well-defined, in the sense that if $u:\ez\to \R$ is 1-~Lipschitz and $K^\e$ is non-negative and $\sum_{i\in\ez}K^\e_i=1$ then $\g^{-1}(K^\e\ast \g(u))$ is still 1-Lipschitz. Indeed, consider $f,g:\R\to \R$ 1-Lipschitz and $t\in(0,1)$. We have 
\begin{multline*}
    \left\lvert \frac d{dx}\g^{-1}(t\,\g(f(x))+(1-t)\,\g(g(x)))  \right \rvert =\left \lvert \frac{t\g'(f(x))f'(x)+(1-t)\g'(g(x))g'(x)}{1-\big(t\,\g(f(x))+(1-t)\,\g(g(x))\big)^2}\right\rvert \\[1.5ex]
    \le \frac{t\g'(f(x))+(1-t)\g'(g(x))}{1-\big(t\,\g(f(x))+(1-t)\,\g(g(x))\big)^2},
\end{multline*}
where we used that $\g'\ge 0$. By convexity,
$(t\g(f(x))+(1-t)\g(g(x)))^2
\le t\g(f(x))^2+(1-t)\g(g(x))^2$ so that
\[
\frac{t\g'(f(x))+(1-t)\g'(g(x))}{1-\big(t\,\g(f(x))+(1-t)\,\g(g(x))\big)^2}
\le
\frac{t\g'(f(x))+(1-t)\g'(g(x))}{1-t\,\g(f(x))^2-(1-t)\,\g(g(x))^2}.
\]
Recalling $\g'=1-\g^2$ we get 
\[
 \Big \lvert\frac d{dx}\g^{-1}(t\,\g(f(x))+(1-t)\,\g(g(x)))\Big \lvert\le  \frac{1 -t\g^2(f(x))-(1-t)\g^2(g(x))}{1-t\g^2(f(x))-(1-t)\g^2(g(x))}=1.
\]
Thus, considering $u,v:\ez\to \R$ 1-Lipschitz,  we can extend them to functions defined on $\R^N$ and by the estimate above   we deduce that $\g^{-1}(t\,\g(u)+(1-t)\,\g(v))$
is 1-Lipschitz on $\ez$. This allows to show that $\g^{-1}(K^\e\ast \g(u))$ is 1-Lipschitz on $\ez$.  Therefore, the term $\sd[\g^{-1}(K^\e\ast \g(u))]$ appearing in \eqref{eq:alg_sdg} is well-defined.

We then prove an estimate on the evolution speed of balls. 
 
\begin{lemma}\label{lem:evol_balls_g}
    Assume that $h(\e)\ge \e^2$.  There exist two constants $R_0,c>0$, such that the following holds. For every $R\in (0,R_0)$ and  for $h$ small enough  it holds
\begin{equation}\label{eq:speed_ball_g}
\sd^+\Big[\gamma^{-1}(K^\e\ast \gamma(|\cdot|-R))\Big]\le |\cdot|-R + c\tfrac hR.
\end{equation}
\end{lemma}

\begin{proof}
Let us fix $0<R_0\le 1$ such that $\g'(R_0)\ge \frac12$. We denote 
\[
v:=K^\e\ast \g(|\cdot|-R).
\]
We start remarking that  the thesis 
follows once we show that in the region $\{ v\ge 0\}$ it holds 
\begin{equation}\label{eq:goal_sd+_g}
    \text{d}^+[\g^{-1}( v)]\le |\cdot|-R+\tfrac{ch}R,
\end{equation}
as then $\text{sd}^+$ is the smallest 1-Lipschitz function lying above $\text{d}^+$ where it is non-negative.

We will in fact prove the following inequality
\begin{equation}\label{eq:subgoal_sd+_g}
    v\le \g(|\cdot|-R+\tfrac{ch}R)  \qquad \text{for all } |i|\in (\tfrac R4, \tfrac R2 ),
\end{equation}
for $R\le R_0$ and $h,\e$ small.

\noindent\textit{Proof of \eqref{eq:speed_ball_g} assuming \eqref{eq:subgoal_sd+_g}.}
If \eqref{eq:subgoal_sd+_g} holds, then we deduce that  
\begin{equation}\label{eq:sign_v}
    v\le 0\quad  \text{ in } B_{R/2}\cap \ez.
\end{equation}
Indeed, from \eqref{eq:subgoal_sd+_g} this inequality already holds in the region $(B_{R/2}\setminus B_{R/4})\cap \ez$. Then, consider a point $i \in B_{R/4} \cap \ez$: there exists $j$ such that $|i - j| \le R/4 + \sqrt{N}\, \e$ and $  |j| \in (R/4,R/4 + \sqrt{N}\, \e)$. By the 1-Lipschitz property, we obtain
\[
v_i \le v_j + \tfrac{R}{4} + \sqrt{N}\, \e \le \g\left(-\tfrac{3R}{4} + \sqrt{N}\, \e + \tfrac{ch}{R}\right) + \tfrac{R}{4} + \sqrt{N}\, \e.
\]
Moreover, assuming $\e,h$  small so that $\g'(-\tfrac{3R}{4} + \sqrt{N}\, \e + \tfrac{ch}{R}) \ge \g'(-R_0) \ge 1/2$, by the equation above and the  convexity of $\g$ in $\{\g \le 0\}$ we get 
\[
\begin{split}
v_i \le \g\left(-\tfrac{3R}{4} + \sqrt{N}\, \e + \tfrac{ch}{R}\right) + \g'\left(-\tfrac{3R}{4} + \sqrt{N}\, \e + \tfrac{ch}{R}\right)\left(\tfrac{R}{2} + 2\sqrt{N}\, \e\right)\\
\le \g\left(-\tfrac{R}{4} + 3\sqrt{N}\, \e + \tfrac{ch}{R}\right).
\end{split}
\]
The right-hand side can be made negative by choosing $\e, h$ sufficiently small. This concludes the proof of \eqref{eq:sign_v}.

With \eqref{eq:sign_v} in hand, we can  follow the argument from Lemma~\ref{lem:balls} to establish \eqref{eq:goal_sd+_g}, and thus complete the proof of the lemma. Specifically, since \eqref{eq:sign_v} holds in $B_{R/2}\cap \ez$, it remains to prove \eqref{eq:goal_sd+_g} for $|i| \ge R/2$. For such $i$, we apply the definition of $\text{d}^+$ from \eqref{eq:sdp}, restricting the infimum to the region where $|j| \in (R/4, R/2)$. Using the bound from \eqref{eq:subgoal_sd+_g}, we obtain an estimate analogous to the right-hand side of \eqref{eq:d+est}. Since $|i|$ is bounded from below, the conclusion follows exactly as in Lemma~\ref{lem:balls}.

\noindent\textit{Proof of \eqref{eq:subgoal_sd+_g}.} To show \eqref{eq:subgoal_sd+_g}, we fix $i\in\ez$ with $|i|\in (R/4,R/2)$  and  estimate 
\begin{equation}
    \begin{split}\label{eq:est_v_g}
    v_i=(K^\e\ast \g(|\cdot|-R))_i &\le  \sum_{|j|\le R/8} K^\e_j \g(|i-j|-R) +  \sum_{|j|\ge R/8} K^\e_j  \Big( \g(|i|-R)+|j|\Big) \\
    &\le \sum_{|j|\le R/8} K^\e_j \g(|i-j|-R) +  \sum_{|j|\ge R/8} K^\e_j \g(|i|-R) +\tfrac {8h}{R},
    \end{split}
\end{equation}
where  in the first inequality we used that $\g$ is 1-Lipschitz and in the second one we used \eqref{Kh}. We need to estimate the term
\begin{equation}\label{eq:first_term}
    \sum_{|j|\le R/8} K^\e_j \g(|i-j|-R).
\end{equation}
Recalling \eqref{eq:Taylor}, for $i,j\in\ez$ with  $|i|\ge R/4$ and $|j|\le R/8$,   {by a Taylor expansion} it holds 
\[
    |i\pm j|\le |i| \pm \frac{i\cdot j}{|i|} + \frac1{|i|-|j|}\frac{|j|^2}2 \le |i| \pm \frac{i\cdot j}{|i|} + \frac 4R |j|^2.
\] 
Therefore, we  can rewrite the term \eqref{eq:first_term} as follows:
\begin{equation}
    \begin{split}
        \sum_{|j|\le R/8} K^\e_j \g(|i-j|&-R) =\frac 12  \sum_{|j|\le R/8} K^\e_j (\g(|i-j|-R) +\g(|i+j|-R) ) \\
    &\le \frac 12  \sum_{|j|\le R/8} K^\e_j (\g(|i|-R -\tfrac{j\cdot i}{|i|} + \tfrac{4}{R} |j|^2 ) +\g(|i|-R+\tfrac{j\cdot i}{|i|} + \tfrac{4}{R}|j|^2) ).\label{eq:mult_g}
    \end{split}
\end{equation}
Let $x,y\in \R$. Since $\g''$ is bounded from above, by a Taylor development we get 
\[
    \g( x+y)  = \g(x) +\g'(x)\, y + \frac{|y|^2}2 \int_0^1\g''(x +  sy)(1-s)\,ds\le \g(x) +\g'(x)\, y + c|y|^2.
\]
We  thus use this estimate to get  
\begin{multline*}
    \g(|i|-R + \tfrac{4}{R} |j|^2 \pm \tfrac{j\cdot i}{|i|}  ) \le  \g(|i|-R+ \tfrac{4}{R} |j|^2) \pm \frac{j\cdot i}{|i|}\, \g'(|i|-R+ \tfrac{4}{R} |j|^2) + c |j|^2\\
    \le    \g(|i|-R  ) \pm \frac{j\cdot i}{|i|} \, \g'(|i|-R+ \tfrac{4}{R} |j|^2) + |j|^2 (c+\tfrac 4{R})
\end{multline*}
where in the last inequality we used the 1-Lipschitz property. Plugging the inequality above in \eqref{eq:mult_g} and using that   $|i|\ge R/2$ and $R\le 1$, we have bounded the term \eqref{eq:first_term} as follows 
\[
\sum_{|j|\le R/8} K^\e_j \g(|i-j|-R)\le \sum_{|j|\le R/8} K^\e_j \g(|i|-R) + \frac{ch}R,
\]
where $c$ is a positive constant.  
Therefore, by \eqref{eq:est_v_g} we get 
\[
(K^\e\ast v)_i \le \g(|i|-R) +\frac{ch}R,
\]
where $c$ denotes a positive constant.  Since $|i|\in(R/4,R/2)$ then $\g'(|i|-R)\ge \g'(-R_0)\ge 1/2$, and therefore 
\[
(K^\e\ast v)_i \le \g(|i|-R) + 2c \frac{h}R\g'(|i|-R)\le \g(|i|-R+ \tfrac{2ch}R),
\]
where the last inequality holds by convexity as long as  $  |i|-R+\tfrac{2ch}R\le 0$, which holds for $h$ small.
\end{proof}

 \subsection{Consistency }
 The link between the algorithm \eqref{eq:alg_sdg} and the motion by mean curvature is provided by   the following variant of  Soner's characterization of  the mean curvature flow  \cite{Soner93}.
\begin{lemma}\label{lem:AC}
Let $\sigma>0$ and consider $E\subset \R^N\times [0,T)$ for  $T>0$, a closed set. Set $d(x,t):= \dist(x,E(t))$ for $x\in \R^N$, $t\in [0,T)$,
where $E(t) := \{x\in \R^N: (x,t)\in E\}$.
Then
  $d$ is  a viscosity supersolution to the heat equation
      in $\comp{E} \cap \{t>0\}$
      if and only if $u(x,t) := \gamma(d(x,t)/\sigma)$ is a
  viscosity supersolution to
  the Allen-Cahn equation:
  \begin{equation}\label{eq:AC}
    \partial_t u = \Delta u - \frac{1}{\sigma^2} W'(u)
  \end{equation}
  in  the same set.
\end{lemma}

\begin{proof}
The proof is standard and follows from the fact that if $\varphi(x,t)$
is  a smooth test function touching from below the graph of $u(x,t)$
at $(\bar x,\bar t)$, then $\varphi(\bar x,\bar t)\in (-1,1)$ so that
in a neighborhood of $(\bar x,\bar t)$, $\tilde\varphi(x,t):= \sigma\gamma^{-1}( \varphi(x,t))$ is a smooth
test function touching from below the graph of $d(x,t)$. Assuming that $d$ is a viscosity supersolution in $\{d>0\}$, one has
\[
\partial_t\tilde\varphi(\bar x,\bar t)\ge \Delta\tilde\varphi(\bar x,\bar t).
\]
Moreover, it holds   $\varphi(x,t)=\gamma(\tilde{\varphi}(x,t)/\sigma)$ and, at  $(\bar x,\bar t)$ we have 
\[\partial_t \varphi=\frac{1}{\sigma}\gamma'\left(\frac{\tilde \varphi}{\sigma}\right)\partial_t \tilde \varphi
\ge \frac{1}{\sigma}\gamma'\left(\frac{\tilde \varphi}{\sigma}\right)\Delta \tilde\varphi,
\]
since $\gamma$ is increasing. Now, note that 
\[
\Delta \varphi = \Div \left[\tfrac{1}{\sigma}\gamma'(\tfrac{\tilde \varphi}{\sigma})\nabla_x \varphi\right]
= \frac{1}{\sigma^2}\gamma''\left(\frac{\tilde \varphi(x,t)}{\sigma}\right)|\nabla_x \varphi|^2
+
\frac{1}{\sigma}\gamma'\left(\frac{\tilde \varphi}{\sigma}\right)\Delta \tilde \varphi.
\]
Since $d(\cdot,\bar t)$ is a distance function, it is standard that 
a smooth contact from below at $x=\bar x$ implies that
$|\nabla_x \varphi(\bar x,\bar t)|= 1$. Thus, using
$\gamma''=W'(\gamma)$ we conclude 
\[
\frac{1}{\sigma}\gamma'\left(\frac{\tilde \varphi}{\sigma}\right)\Delta \tilde\varphi = 
\Delta\varphi - \frac{1}{\sigma^2}W'\left(\gamma\left(\frac{\tilde \varphi}{\sigma}\right)\right)
= \Delta\varphi - \frac{1}{\sigma^2}W'(\varphi)
\]
The thesis follows. The reverse implication is shown in the same way.
\end{proof}
As a consequence of the previous lemma, if $u$ is a viscosity
super solution to~\eqref{eq:AC} in $\{u>0\}$, it is  a subsolution in $\{u<0\}$, and the set $\{u=0\}$ does not develop fattening, then $\{u\le 0\}$
evolves by mean curvature. 
We can show, however, a stronger result.
\begin{lemma}\label{lem:AC-heat}
Let $E\subset \R^N\times [0,T)$, $T>0$, be a closed set
and set $d(x,t):= \dist(x,E(t))$ for $x\in \R^N$, $t\in [0,T)$,
where $E(t) := \{x\in \R^N: (x,t)\in E\}$,
and $v(x,t)=\gamma(d(x,t)/\sigma)$ for some given $\sigma>0$.
Then if $v$ is a viscosity supersolution to the
heat equation in $\comp{E} \cap \{t>0\}$, then 
also $d$ is (and hence, by Lemma~\ref{lem:AC} $v$
is a supersolution to Allen-Cahn in the same set), 
and $E$ is a 
superflow for the mean curvature motion. 
\end{lemma}

\begin{proof}
Indeed, if $\varphi$ touches the graph
of $d$ from below at some point $(\bar x,\bar t)\in\{ d>0 \}$,
we have $\tilde\varphi:=\gamma(\varphi/\sigma)$ touches
the graph of $u$ from below at the same point and
\[
\partial_t\tilde\varphi(\bar x,\bar t)\ge \Delta\tilde\varphi(\bar x,\bar t).
\]
In addition,  $|\nabla_x\varphi(\bar x,\bar t)|=1$.
Hence, as before, at $(\bar x,\bar t)$ it holds 
\[
 \partial_t \varphi \ge \Delta \varphi + \frac{1}{\sigma}\frac{\gamma''(\varphi/\sigma)}{\gamma'(\varphi/\sigma)}
\]
Note that $\gamma''/\gamma' = W'(\gamma)/\sqrt{2W(\gamma)}
= (\sqrt{2W})'(\gamma)$, that is, with our choice, 
$-2\gamma$, hence
\[
\partial_t \varphi \ge \Delta \varphi - \tfrac{2}{\sigma}\tilde\varphi.
\]
We recall that, still with our choice, $\gamma(s) = \tanh(s )$, so that the last term
is $\frac{2}{\sigma}\tanh(\tfrac{1}{\sigma} d(\bar x,\bar t))$
and thus  $d$ is a   {viscosity} supersolution of
\[
\partial_t d\ge \Delta d - \sqrt{2}\frac{d}{\sigma^2} \quad \text{ in }\{d>0\}.
\]
Yet this implies that $d$ is also a supersolution of the
heat flow in $\{d>0\}$.   {Indeed, for every $s>0$ the function $d$ is a viscosity supersolution to $\partial_t d\ge \Delta d - \sqrt{2}\frac{s}{\sigma^2}$ in $\{0<d<s\}$. By classical computations (see for instance \cite[p.~366]{Soner93}), it then holds that $\partial_t d\ge \Delta d - \sqrt{2}\frac{s}{\sigma^2}$ in $\{d>0\}$, in the viscosity sense.} 
The claim follows.
\end{proof}

Starting from the algorithm \eqref{eq:alg_sdg}, the idea is then to define 
\[
v^{k,\e}:=\g(u^{k,\e})
\]
and characterize the  evolution of the function $v^{k,\e}$. Indeed, for every $j$ in the region where $\sd^{k+1,\e}>0$ it holds  
\[
  \frac{v^{k+1,\e}_j-v^{k,\e}_j}{h} \ge \frac{\g\Big( \g^{-1} ( K^\e \ast v^{k,\e})_j\Big)-v^{k,\e}_j}{h}    = \left(\frac{K^\e- \delta}{h}*v^{k,\e}\right)_j,
\]
where $\delta = (\delta_j)_{j\in\e\Z^N}$ is defined by $\delta_0=1$,
$\delta_j=0$ for $j\neq 0$. We write this inequality as
\[
  \frac{v^{\e}(t+h)-v^{\e}(t)}{h} \ge \frac{K^\e- \delta}{h}*v^{\e}(t)
\]
for any $t\ge 0$ where $v^\e(t) := v^{[t/h],\e}$.
Thanks to the estimate in Lemma~\ref{lem:evol_balls_g},  reasoning as in  Section~\ref{sec:proof} we obtain the same compactness
as in~\cite[Prop.~4.4]{CMP17} for the function $d^\e:=u^{[t/h],\e}$. Therefore, up to a subsequence,  $d^\e$ converges
locally uniformly in space and for all times but a countable number (and until an extinction time $T^*\in (0,+\infty]$), to
$d(x,t)=\dist(x,E(t))-\dist(x,\comp{A}(t))$, where $E$ is the Kuratowski limit of $\{(j,t): d^\e(t)_j\le 0\}$
and ${A}\subset E$ the complement of the Kuratowski limit of $\{(j,t): d^\e(t)_j\ge 0\}$. Moreover, up to a subsequence, by definition the function $v^\e$ converges locally uniformly in space and for all times but a countable number to $\g(d)$, with $d$ as above.

From this, we can follow the proof of Section~\ref{sec:proof} to conclude that 
\[
\partial_t \g(d)\ge \Delta \g(d)
\]
in $E^\complement$, in the distributional sense. Using Lemma~\ref{lem:AC-heat} we then conclude that $\g(d)$ is a supersolution of the Allen-Cahn equation, and thus $d$ is a supersolution of the heat equation. This means that $E$ is a superflow. We have shown the following result.
\begin{theorem}\label{th:mainnonlin}
  Let $E^0\subseteq \R^N$ and consider the algorithm \eqref{eq:alg_sdg}. As $\e\to 0$, the function $d^\e$, defined for $t\ge 0$ and $i\in\e\Z^N$
  by $d^\e(t)_i = u^{[t/h],\e}_i$, converge up to subsequences,
  for almost all times and locally uniformly in space to a function
  $d(x,t)$ such that $d^+=\max\{d,0\}$ is the distance function
  to a supersolution of the (generalized)
  mean curvature flow starting from $E^0$,
  and $d^-$ is the distance function to a supersolution starting from
  $\comp{(E^0)}$.
\end{theorem}

\section{Applications: discrete heat flow, implicit Laplacian}\label{sec:applic}
In this section we list two examples of numerical schemes whose associated resolvent operator is a kernel satisfying the assumptions listed in Section \ref{sec:Kernel}. Note that the kernel associated to the explicit Euler scheme has been investigated in Section \ref{sec:algo_explheat}.
\subsection*{Implicit Laplacian scheme}
We start by considering the fully discrete implicit Euler's scheme that reads as follows. The idea is to use as kernel $K^\e$ 
associated to the resolvent operator for 
\begin{equation}\label{eq:impl_scheme}
    u - \tau\Delta_\e u=g\quad \text{ in }\ez.
\end{equation}
Note that if $g\in\ell^2(\ez)$, there exists a unique solution $u\in\ell^2(\ez)$ of the equation above, and a comparison principle holds. 
Rewriting \eqref{eq:impl_scheme} in the Fourier variable we deduce that  $K^\e$ satisfies
\[
    \widetilde{K^\e}(\xi)=\frac1{1-\tau\widetilde{\Delta_\e}},
\]
where, with a slight abuse of notation, $\widetilde{\Delta_\e}$ denotes the Fourier transform of the kernel associated to $\Delta_\e$.
Recalling that
\begin{equation}\label{eq:ker_Delta}
    \widetilde{\Delta_\e}(\xi)=\frac{2}{\e^2}\sum_{k=1}^N (\cos(2\pi\e  \xi\cdot e_k)-1),
\end{equation}
see  Remark \ref{rmk:expl_lapl}, we obtain
\[
    \widetilde{K^\e}(\xi)=\frac1{1-\frac{2\tau}{\e^2}\sum_{k=1}^N (\cos(2\pi\e  \xi\cdot e_k)-1)}.
\]
Note that the regularity of the Fourier transform of the kernel implies that $K^\e$ is rapidly decaying, so that $K^\e\ast d$ is well-defined when $d\in Lip(\ez)$.
Therefore, we can define our algorithm as 
\begin{equation*}
    \begin{split}
    &u^{k+1,\e} = K^\e\ast d^{k,\e}\\
    &\sd^{k+1,\e}=\sd[u^{k+1,\e} ],
\end{split}
\end{equation*}
where the first equation above implies that
\[
     u^{k+1,\e} - \tau\Delta_\e u^{k+1,\e}=d^{k,\e}\quad \text{ in }\ez.
\]
Positivity of $K^\e$ follows from the comparison principle holding for   \eqref{eq:impl_scheme},  while symmetry follows from the symmetry of $\widetilde{K^\e}$. Moreover, since  $\widetilde{K^\e}(0)=1$, it holds $\sum_{j\in\ez} K^\e_j=1$. It remains to check that the time-step $\tau(\e)$ coincides with the natural time-step $h(\e)$ defined in \eqref{Kh}.
In order to compute $h$, we recall (\ref{Kh}') i.e. $h=-\frac{\Delta \widetilde{K^\e}(0)}{8\pi^2 N}$. By a straightforward computation we check that 
\[
    \Delta \widetilde{K^\e}(0)=-8N\pi^2\tau,
\]
hence $h=\tau$. In particular, \eqref{Kh} is satisfied. Lastly, using the explicit expression of $\widetilde{K^\e}$, one can check that \eqref{eq:approxsymbol}  holds.

We conclude by noting that, while the hypothesis on the kernel $K^\e$ are satisfied for any $\tau(\e)=h(\e)$ going to 0 as $\e\to0$, Corollary \ref{lem:balls} requires $h(\e)\gtrsim \e^2$. In conclusion, if $h(\e)\gtrsim \e^2$,  the theory developed in the previous section applies and Theorem \ref{th:variantexplicit} holds.

\subsection*{Discrete heat flow scheme}
Another interesting example is the discrete heat flow. We  define $S_\e(t)$ the semigroup generated by the discrete Laplacian, that is $ S_\e(\cdot)u=v$ with $v$ solving
\begin{equation}\label{eq:heat_semidiscr}
    \begin{cases}
     v_t - \Delta_\e v=0\quad \text{ in }\ez\times[0,+\infty)\\
     v(0)=u,
   \end{cases}   
\end{equation}
for $u\in\ell^2(\ez)$.
We consider a parameter $\tau=\tau(\e)$ with $\tau\to 0$, and define the kernel  $K^\e$ as satisfying $S_\e(\tau)u=K^\e\ast u$ for every $u\in\ell^2(\ez)$. 
To find a solution to the equation \eqref{eq:heat_semidiscr}, we pass to the Fourier variable in space 
\[
    \dot{\widetilde{v}}(\xi,t)- \widetilde{\Delta_\e}(\xi) \widetilde v(\xi,t)=0, \quad \text{ in }[0,\tfrac1\e)^N\times [0,+\infty),
\]
where  $\widetilde{\Delta_\e}$ is defined in \eqref{eq:ker_Delta}.
Therefore, we find $\widetilde{S_\e(\tau)u}(\xi)=\widetilde{v}(\xi,\tau)=\widetilde{K^\e}\widetilde{u}(\xi)$ where 
\[
    \widetilde{K^\e}(\xi)=\exp\left(\frac{2\tau}{\e^2}\sum_{k=1}^N (\cos(2\pi\e  \xi\cdot e_k)-1)\right).
\]
Then positivity of $K^\e$ follows from the comparison principle, symmetry and condition~\eqref{Kone} (equivalent to $\widetilde{K^\e}(0)=1$) are easily checked, while an explicit computation of $\Delta \widetilde{K^\e}(0)$, arguing as for the previous example, yields $h=\tau$, and thus also \eqref{Kh} is satisfied. Similarly, condition \eqref{eq:approxsymbol} is a simple computation using that $h=\tau$.

Again, note that the regularity of  $\widetilde{K^\e}$ implies that $K^\e$ is rapidly decaying, so that $K^\e\ast d$ is well-defined when $d\in Lip(\ez)$. Therefore, we can define our algorithm
\begin{equation*}
    \begin{split}
    &u^{k+1,\e} = K^\e\ast d^{k,\e}\\
    &\sd^{k+1,\e}=\sd[u^{k+1,\e} ].
\end{split}
\end{equation*}
As before,  if $h(\e)\gtrsim \e^2$,  the theory developed in the previous section applies and Theorem~\ref{th:variantexplicit} holds.

\section{Numerics}\label{sec:numerics}
We now present some numerical experiments performed using the algorithm \eqref{eq:defalgo}:
we have implemented~\eqref{eq:impl_scheme}, using
\texttt{libfftw3}\footnote{See~\cite{fftw} and \url{https://www.fftw.org/}.} in C,
in order to efficiently solve the implicit equation which is
diagonalized in Fourier   {domain} (we use in practice a Neumann Laplace operator which is diagonalized with the DCT).

\begin{figure}[htb]
\includegraphics[width=0.5\textwidth]{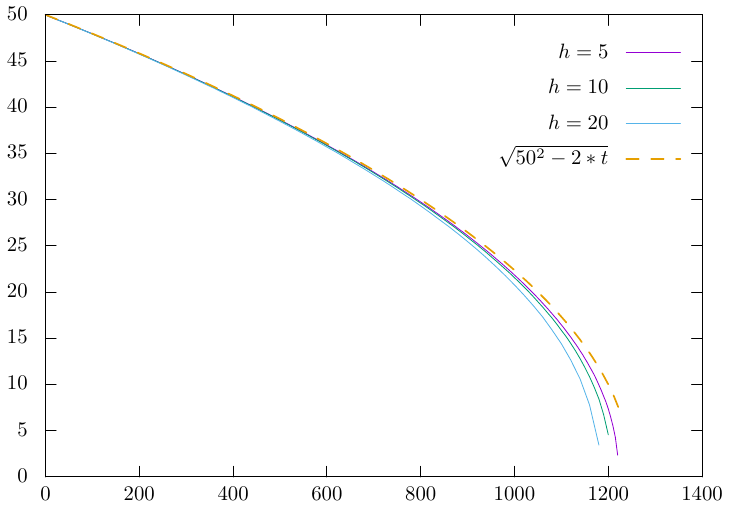}
\caption{Evolution of the radii for an initial disk of radius $50$,
$\e=1$, and $h=5, 10, 20$ (the explicit scheme needs $h=.25$).
The precision remains quite good for small values of $h$.
Compare with Figure~\ref{fig:Explicit}, right.}\label{fig:Implicit}
\end{figure}

\begin{figure}[htb]
\includegraphics[width=0.24\textwidth]{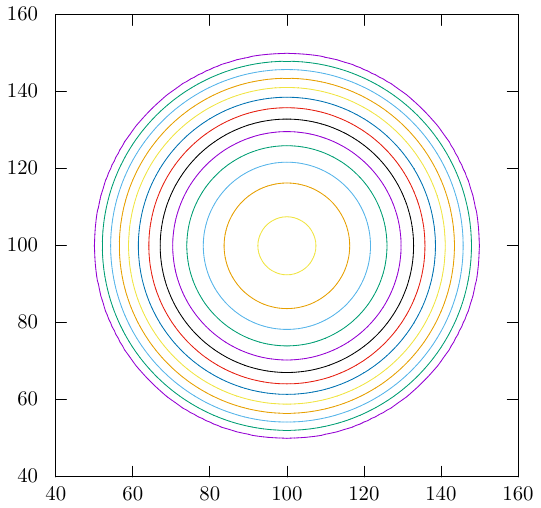}\hfill 
\includegraphics[width=0.24\textwidth]{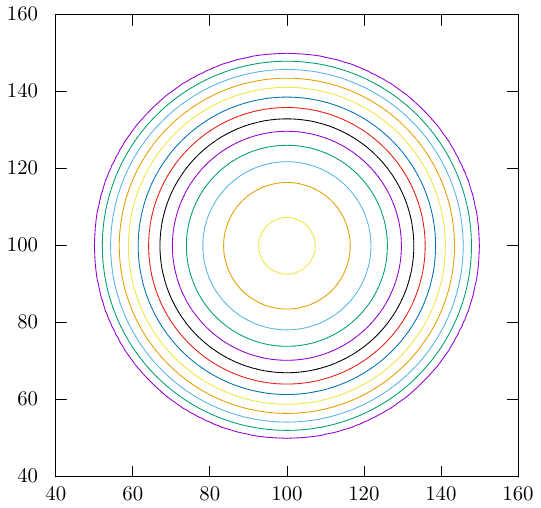}\hfill 
\includegraphics[width=0.25\textwidth]{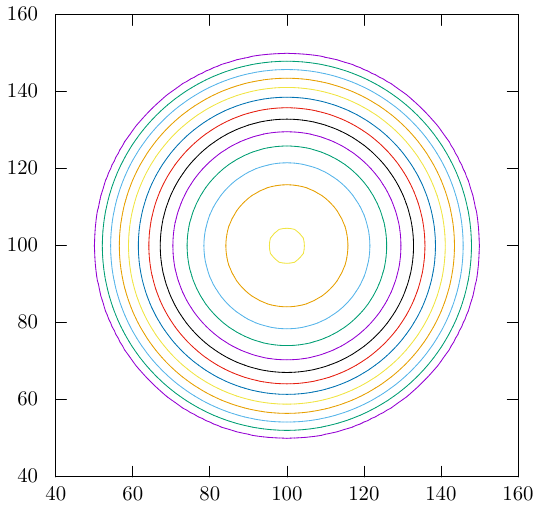}\hfill 
\includegraphics[width=0.25\textwidth]{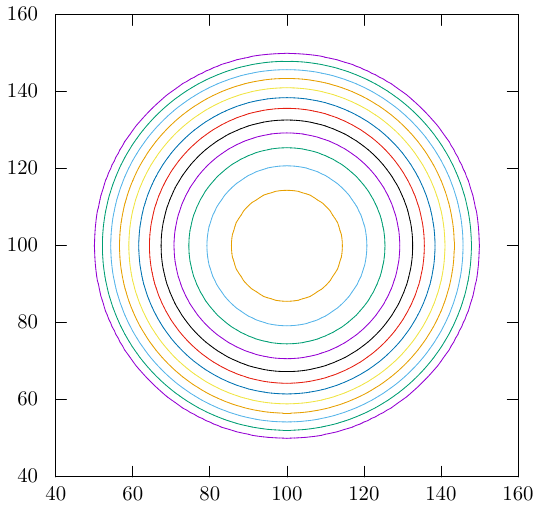}
\caption{Evolutions of a disk of radius $50$
with $h=.25$ (explicit scheme as in Fig.~\ref{fig:Explicit})
and $h=5,10,20$ (implicit scheme~\eqref{eq:impl_scheme}).}\label{fig:ImplicitDisks}
\end{figure}

We also compared the flows starting from the Mask pattern of
Figure~\ref{fig:ExplicitMask}, running the explicit scheme ($h=.25$) and the implicit one for $h=5$: Figure~\ref{fig:ImplicitMasks} shows that the differences are hardly visible except at the smaller scales.

\begin{figure}[htb]
\includegraphics[width=0.24\textwidth]{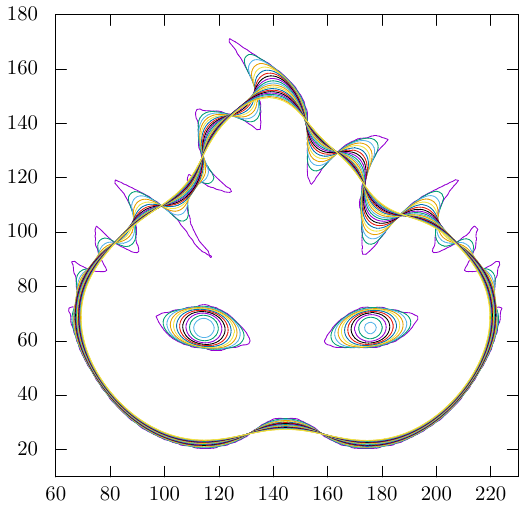}\hfill 
\includegraphics[width=0.24\textwidth]{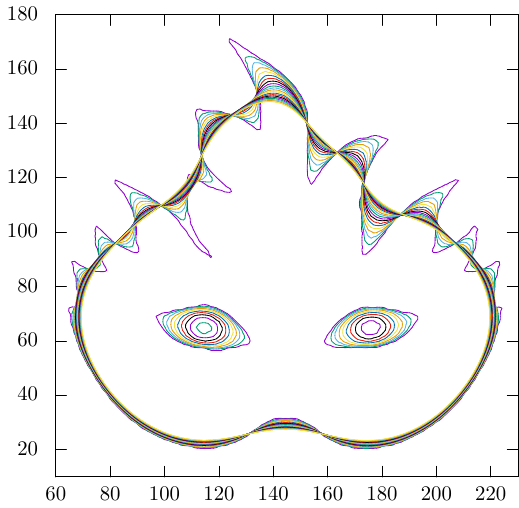}\hfill 
\includegraphics[width=0.25\textwidth]{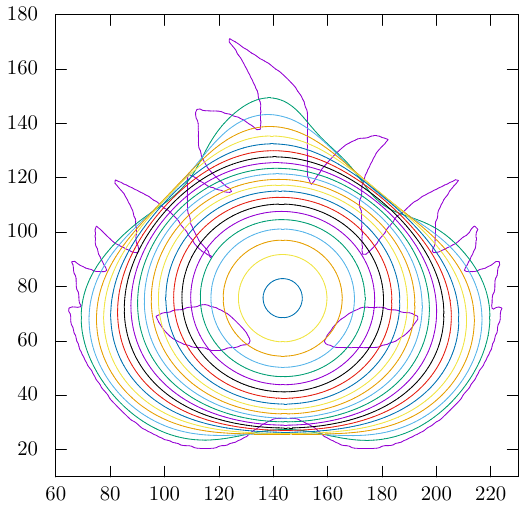}\hfill 
\includegraphics[width=0.25\textwidth]{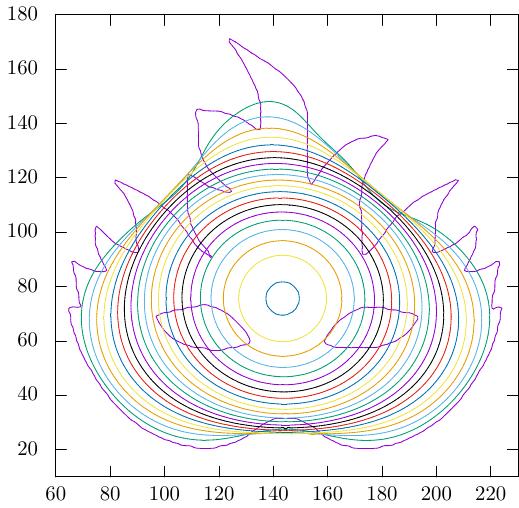}
\caption{Evolutions of the mask pattern, comparison between the explicit and implicit schemes. Left two plots: explicit vs implicit ($h=5$) evolution, times $0,5,\dots,100$; Right two plots: same for times $0,100,200,\dots$}\label{fig:ImplicitMasks}
\end{figure}

\section{Concluding remarks}\label{sec:concluding}

\subsection{Restricting the computations on a strip}\label{sec:strips}
In this subsection we consider the following modification of the redistancing operations $\ud^\pm,\sd^\pm$ defined in \eqref{eq:sdp}, \eqref{eq:sdm}. Consider a (fixed) positive parameter $M>0$, and assume that $M>\e$. Given $E\subseteq \ez$, we  define $S_M E:=\{ i\in E : \exists k\in \ez\setminus E, |i-k|\le M \}$.

  {We  consider $u:\ez\to\R$    1-Lipschitz. For every  $i\in\ez$  we set:}
\begin{equation}\label{eq:sdM}
  \begin{cases}
    \ud^{M,+}[u]_i := \inf_{j\in S_M\{u<0\}  }   {(} u_j + |j-i|  {)}\,,\\
    \sd^{M,+}[u]_i := \sup_{j\in S_M\{u\ge 0\}}   {(} \ud^{M,+}_j - |j-i|   {)}
  \end{cases}
\end{equation}
and analogously  $\ud^{M,-},\sd^{M,-}$.  Clearly, the %two
functions defined above are both 1-Lipschitz. By the same arguments used in Lemma \ref{lem:defsd}, one can show that $\ud^{M,+}[u] = u$ in $S_M\{u<0\}$ and $\ud^{M,+}[u] \ge  u$ in $\{ u\ge 0\}$, thus $\sd^{M,+}[u] = \ud^{M,+}\ge u$ in $S_M\{u\ge 0\}$ and $\sd^{M,+}[u] \le  \ud^{M,+}[u] \le u$ in $S_M\{u< 0\}$.
In particular,
\begin{equation}\label{eq:sdMu}
    \sd^{M,+}[u] \le u \text{ in } S_M\{u< 0\}, \quad \sd^{M,+}[u] \ge u \text{ in } S_M\{u\ge  0\},
\end{equation}
showing a local comparison between $\sd^{M,+}[u]$ and $u$ in the ``strip'' 
\[
    S_M\{u< 0\}\cup S_M\{u\ge 0\}, 
\]
instead of the global one satisfied by $\sd^+[u]$. As for definitions \eqref{eq:sdp}, \eqref{eq:sdm}, we note that   \eqref{eq:sdM} is coherent (i.e., the inequalities in \eqref{eq:sdMu} hold)  even if $\{u<0\}=\emptyset$ or $\{u>0\}=\emptyset$.

Furthermore, we can show that this  redistancing operation relates well with $\sd^+[u]$, at least in the strip.
\begin{lemma}\label{lem:compsdsdM}
Given   {$u:\ez\to\R$} 1-Lipschitz, for every $ i\in S_M\{u< 0\}\cup S_M\{u\ge 0\}$ it holds
\begin{align}
    &\sd^{M,+}[u]_i\le \sd^{+}[u]_i +  \frac{4N}M\e^2,\label{eq:sdMone}\\
    &|\sd^{M,+}[u]_i-d_i|\le   {(4\sqrt N+1)}\e,\label{eq:sdMtwo}
\end{align}
where $d=\min_{j\in \{u\ge 0\}}|i-j|-\min_{j\in \{u<0\}}|i-j|. $
\end{lemma}

\begin{proof}
   Let $i\in S_M\{u\ge 0\}$, then $\sd^{M,+}[u]_i=\ud^{M,+}[u]_i$ and $\sd^{+}[u]_i=\ud^{+}[u]_i$. Moreover, by definition it holds $\ud^+[u]\le \ud^{M,+}[u]$ on $\ez$. We thus estimate, using that $u_k-u_j\le |k-j|$,
   \begin{multline*}
        \ud^{M,+}[u]_i-\ud^{+}[u]_i\le \sup_{j\in \{u< 0\}} \inf_{k\in S_M\{u< 0\}}   {(} |k-j| +|k-i|-|j-i|   {)}\\
        \le  \sup_{j\in \{u< 0\}\setminus S_M\{u<0\} } \inf_{k\in S_{3M/4}\{u< 0\}\setminus S_{M/4}\{u< 0\}}    {(}
        |k-j| +|k-i|-|j-i|   {)} \le \frac{4N}{M} \e^2
   \end{multline*}
where we used 
\[
\sup_{j\in S_M\{u< 0\}}\inf_{k\in S_M\{u< 0\}}   {(} |k-j| +|k-i|-|j-i|   {)}=0, 
\]
and for the last inequality we reasoned as for the last inf in  equation \eqref{eq:d+est} (see the proof of \eqref{eq:triangle}). On the other hand, for $i\in S_M\{u< 0\}$ by definition of $\sd^+[u]$ and the inequality above, we have 
\[ \sd^{+}[u]_i\ge \sup_{j\in S_M\{ u\ge 0\}}   {(} \ud^+[u]_j-|j-i|   {)}\ge \sd^{M,+}[u]_i - \frac{4N}{M} \e^2 \]
and we conclude \eqref{eq:sdMone}.
Concerning \eqref{eq:sdMtwo}, the bound
\[
\sd^{M,+}[u]_i\le d_i + \sqrt N \e
\]
follows from  \eqref{eq:compdist}, \eqref{eq:compdistm}.
The other estimate can be proved as follows.  Consider $i\in S_M\{u< 0\}$ and let $\bar j\in \{u\ge 0\}$ realising the $\min$ in the definition of $d_i$. Note that $\bar j\in S_M\{u\ge 0\}$, thus 
\[ \sd^{M,+}[u]_i-d_i= \sd^{M,+}[u]_i + |i-\bar j| \ge \ud^{M,+}[u]_{\bar j} \ge 0.\]
If instead $i\in S_M\{u\ge  0\}$, by the previous arguments $\sd^{M,+}[u]_i=\ud^{M,+}[u]_i\ge \ud^+[u]_i$. Thus the conclusion follows using \eqref{eq:sdMone}, \eqref{eq:compdist} and \eqref{eq:compdistm}.
\end{proof}

We can then modify the proposed algorithm \eqref{eq:defalgo} as follows: given $\sd^{0,M,\e} $ defined upon an initial set $E^0$, we set for all $k\in\N$:
\begin{equation}\label{eq:defalgoM}
    \sd^{k+1,M,\e} = \sd^{M,+}[K^\e \ast \sd^{k,M,\e} ].
\end{equation}
Following  Section \ref{sec:kernels}, we can establish the following result.
\begin{theorem}\label{th:strip}
Theorem \ref{th:mainexplicit} holds for $\sd^{k,M,\e}$ replacing $\sd^{k,\e}$.
\end{theorem}
We now proceed to outline the main step of the proof of Theorem \ref{th:strip}, which closely follows the proof of Theorem \ref{th:mainexplicit}, highlighting only the key differences.

Note that the Lemma \ref{lem:compsdsdM} and Corollary \ref{lem:balls2} imply, in particular, that there exists a suitable constant $C\ge 1$ such that for every $R<M$
\[ \sd^{M,+}[K^\e\ast (|\cdot|-R)]\le |\cdot | - R + \frac CR \e^2 \quad \text{ in } B_{2R}\cap \ez,  \]
as soon as  $\e/R$ is small enough. This, in turn, yields an estimate on the evolution speed of balls under the scheme assuming $h(\e)\ge \e^2$: one step of the algorithm applied to $B_R$ contains $B_{R-\frac CR h}$, as long as $R<M$, $\e/R$ is small enough and $R-\frac CR h\ge R/2$. By comparison and translation invariance,  one concludes that one step of the algorithm applied to $B_R$, for any $R>0$,  contains\footnote{Where $R\wedge M:=\min\{R, M\}$} $B_{R-\frac C{R\wedge M} h}$, as long as  $\e/R$ is small enough and $R-\frac C{R\wedge M} h\ge R/2$.

Reasoning as in the proof of \cite[Cor.~4.13]{CDGM-crystal} and Section \ref{sec:consistexplicit}, one can show that if for some $k$, $\sd^{k,M,\e}_i\ge R>0$ at $i\in\ez$, then
\[
  \sd^{\ell,M,\e}_i \ge R- \frac{C}{R\wedge M}(\ell - k )h \implies \sd^{\ell,M,\e}_i \ge \sd^{k,M,\e}_i- \frac{C}{R\wedge M}(\ell - k )h
\]
for $\ell \ge k$, $\e/R$ small enough  and as long as the right-hand side is larger than $R/2$. This allows to follow the arguments of Section \ref{sec:consistexplicit} and show convergence (up to subsequences) as $\e\to 0$ of the sequence 
\[  \{ \ud^{M,\e}(t):= \sum_{i\in\ez}\sd^{[ t/h],M,\e}_i \chi_{[0,\e)^N}\} \to \ud^M,\]
locally uniformly on $\R^N$ and for a.e. $t\in (0,+\infty)$. 
Thanks to \eqref{eq:sdMtwo}, the  function $\ud^M$ satisfies
\begin{align*}%\label{eq:distMEt}
  &(\ud^M)^+(\cdot,t) = \dist(\cdot,E^M(t)) \qquad \text{ on } \{0<\ud^M<M\}, \\
  &(\ud^M)^-(\cdot,t) = \dist(\cdot,\comp{(A^M)}(t))\qquad \text{ on }\{-M<\ud^M<0\},
\end{align*}
where $E^M,A^M$  are space-time tubes satisfying the same properties of $E,A$ defined in \eqref{eq:distEt}.
Let us set 
\[ \ud(\cdot,t)=\begin{cases}
  \dist(\cdot,E^M(t)) \qquad &\text{ on } (\R^N\times (0,+\infty))\setminus E^M(t)\\
  -\dist(\cdot,\comp{(A^M)}(t))\qquad &\text{ on } A^M(t)\\
  0 \qquad &\text{ otherwise.}
\end{cases} \]
Then, employing the inequalities between $\sd^{\e}[u]$ and $u$ holding in a strip, one can follow the argument of Section \ref{sec:proof} to show that  $\partial_t \ud\ge \Delta \ud$ on $\{0<\ud<M\}$, and  $\partial_t \ud\le \Delta \ud$ on $\{-M<\ud<0\}$, 
in the distributional (and viscosity) sense. By standard arguments, this in turn implies that the two inequalities are satisfied, respectively, on $(\R^N\times (0,+\infty))\setminus E^M $ and  $A^M $.
The other properties follow reasoning as in Section \ref{sec:proof}, concluding the proof of Theorem \ref{th:strip}.

\subsection{Relating our redistancing operator  with fast solvers for distance functions}
The main drawback of our proposed scheme is the computation cost of the redistancing operator $\sd$, although the previous Section \ref{sec:strips} showed how to reduce computations in practice while maintaining consistency.
Ideally, one would use well-known fast numerical solvers for computing the distance function to a given set in our setting. A natural candidate is the class of fast-marching techniques \cite{RouyTourin,Set,OshSet,ElseyEsedoglu14,Saye14}.  
However, a major challenge lies in the lack of precise error estimates and regularity properties for the distance function computed with these methods, which are crucial in our framework. In particular, it is unclear whether the key inequality in item (2) of Lemma \ref{lem:defsd} holds (up to controlled errors) when using such schemes.

On the one hand, the error can generally be estimated as $O(\sqrt{\e})$, as shown in \cite[Sec.~9]{Mon-notes}. On the other hand, a precise convergence rate for the scheme is only required in Lemma \ref{lem:balls}, where we compare the redistancing operator applied to a distance function that is smooth outside a single point (the signed distance to a ball). In this specific context, higher-order implementations of the fast-marching scheme \cite{SetSch,AhmBakMcLRen} could provide interesting insights.

\subsection{Partitions}
Our redistancing can be used to define a scheme for the evolution of partitions, similarly
to~\cite[Sec.~5.3]{EsedogluRuuthTsai10}
and~\cite{ElseyEsedogluSmereka11, EseOtt}. 
Thanks to the linearity of the
convolution part, one way to implement
this is to track at each iteration the
non-signed (non-negative) distance function $(u^\ell_i)_{i\in\ez}$
to the phase $\ell\in\{1,\dots, L\}$, 
and alternatively (i) diffuse
each $u^\ell$ with the convolution kernel, (ii)
recompute $u^\ell$ as the positive distance to
the negative values of $u^\ell_i - \max\{u^{\ell'}_i: \ell'\neq \ell\}$. This seems to produce results
similar to the schemes in~\cite{EsedogluRuuthTsai10,ElseyEsedogluSmereka11}, yet we have
no convergence proof (when two
phases only are present, this is the same as the scheme
studied in this paper, yet the behavior near points where more than two phases meet would have to be analyzed).

\appendix
\section{An estimate on the decay of the Fourier
  transform of discretized rapidly decaying functions}
 Let $\eta\in C_c^\infty(\R^N)$, $\e \in (0,1)$,
we let $\eta^\e_j = \eta(j)$ for $j\in\e\Z$ (that is, $\eta^\e=\eta_{|\e\Z^N}$), and
we consider the $(1/\e)$-periodic Fourier transform
\[\widehat{\eta^\e}(\xi) = \e^N \sum_{j\in\e\Z^N} \eta(j)e^{-2\pi i j\cdot\xi}.\]
\begin{lemma}\label{lem:uniformSchwartz}
  For any $p>N$, there exists $C$ such that for any $\e\in (0,1]$:
  \begin{equation}\label{uniform}
    \sup_{\xi\in [-\frac{1}{2\e},\frac{1}{2\e}]^N}|\widehat{\eta^\e}-\widehat{\eta}|\le C\e^p
  \end{equation}
  and  for any $\xi\in [-\frac{1}{2\e},\frac{1}{2\e}]^N$ it holds 
  \begin{equation}\label{Schwartz}
       |\widehat{\eta^\e}(\xi)|\le \frac{C}{1+|\xi|^p}.
     \end{equation}
\end{lemma}
\begin{proof}
   By Poisson's summation formula, $\widehat{\eta^\e}$ is the $(1/\e)$-periodic version of $\widehat{\eta}$,
   that is:
   \[
     \widehat{\eta^\e}(\xi) = \sum_{j\in \frac{1}{\e}\Z^N} \widehat{\eta}(\xi+j).
   \]
   Since $\widehat{\eta}$ is in the Schwartz class,
   for $p\ge 1$, there exists $C  {>0}$ such that for $\xi\in\R^N$,
   \begin{equation} \label{eq:schweta}
     \widehat{\eta}(\xi)\le \frac{C}{1+|\xi|^p}.
   \end{equation}
   Consider $\xi \in [-\frac{1}{2\e},\frac{1}{2\e}]^N$, that is, $|\xi|_\infty\le \frac{1}{2\e}$.
   We write:
   \[
     \widehat{\eta^\e}(\xi) = \widehat{\eta}(\xi)+ \sum_{j\in \frac{1}{\e}\Z^N, j\neq 0} \widehat{\eta}(\xi+j)
   \]
   and
   \[
     e:=\left|\sum_{j\in \frac{1}{\e}\Z^N, j\neq 0} \widehat{\eta}(\xi+j)\right|
     \le \sum_{j\in \frac{1}{\e}\Z^N, j\neq 0} \frac{C}{1+|\xi+j|^p} \le
     \sum_{j\in \frac{1}{\e}\Z^N, j\neq 0} \frac{C}{|\xi+j|_\infty^p} 
   \]
   possibly increasing the constant $C$ in the last term. We have
   \[
     |\xi+j|_\infty^p \ge 2^{1-p}|j|_\infty^p - |\xi|_\infty^p.
   \]
   If  $|j|_\infty\ge \frac{1}{\e}$ then  $|\xi|_\infty\le \frac{1}{2\e}\le \frac{1}{2}|j|_\infty$,
   hence
   \[
     2^{1-p}|j|_\infty^p - |\xi|_\infty^p\ge 2^{-p}|j|_\infty^p \ge \frac{1}{2} (2^{-p}|j|_\infty^p + |\xi|_\infty^p).
     \]
   Hence, assuming $p>N$,
   \begin{equation}\label{eq:app-err1}
   \begin{split}
       e \le \sum_{j\in \frac{1}{\e}\Z^N, j\neq 0} \frac{C}{2^{-p}|j|_\infty^p+|\xi|_\infty^p}
     =  \sum_{j\in \Z^N, j\neq 0} \frac{C2^p\e^p}{|j|_\infty^p+2^p\e^p|\xi|_\infty^p} \\
      \le
     C\sum_{j\ge 1}   \frac{2^p\e^p j^{N-1}}{j^p+2^p\e^p|\xi|_\infty^p}   \le C\e^p
   \end{split}
   \end{equation}
   (increasing again $C$), which shows~\eqref{uniform}.
In addition,
   one can estimate the sum as follows: if $p>N$, for any $a\in (0,1]$,
   \[
     \sum_{j\ge J_p} \frac{j^{N-1}}{a^p+j^p} \le \int_0^\infty \frac{t^{N-1}}{a^p+t^p}dt
   \]
   for some $J_p\ge 1$ (which is such that $t\mapsto \frac{t^{N-1}}{a^p+t^p}$ is
   decreasing for $t\ge J_p$, hence one may use $J_p=(N-1)^{1/p}$). 
   Then,
   \[
     \int_0^\infty \frac{t^{N-1}}{a^p+t^p}dt = \frac{1}{a^p}\int_0^\infty \frac{t^{N-1}}{1+(t/a)^p}dt
      = a^{N-p} \int_0^\infty \frac{s^{N-1}}{1+s^p}ds.
   \]
   It follows:
   \begin{equation*}
     e\le \frac{C}{|\xi|^p} + C 2^p\e^p (2\e|\xi|_\infty)^{N-p} =
     \frac{C}{|\xi|^p} + \e^N\frac{C}{|\xi|^{p-N}}.
   \end{equation*}
     {Since $\xi\in[-\frac1{2\e},\frac1{2\e}]^N$, we deduce $e\le C/|\xi|^p$ (up to increasing $C$), which we combine with  \eqref{eq:app-err1} and \eqref{eq:schweta} to   deduce~\eqref{Schwartz}.}
 \end{proof}

\section*{Acknoledgements}
	The authors would like to thank the anonymous referees for their careful reading and suggestions that helped improve the manuscript.

A.C.~acknowledges the support  of the ``France 2030'' funding ANR-23-PEIA-0004 (``PDE-AI''). D.D.G. is funded by the European Union: the European Research Council (ERC), through StG ``ANGEVA'', project number: 101076411. Views and opinions expressed are however those of the authors only and do not necessarily reflect those of the European Union or the European Research Council. Neither the European Union nor the granting authority can be held responsible for them. M.M. is partially supported by PRIN 2022 Project ``Geometric Evolution Problems and Shape Optimization (GEPSO)'', PNRR Italia Domani, financed by European Union via the Program NextGenerationEU, CUP\_D53D23005820006. He wishes to warmly thank the hospitality of CEREMADE, where part of this research was conducted. D.D.G. and M.M. are  members of the Italian National Research Group GNAMPA (INDAM).

 \printbibliography

\end{document}